\documentclass[12pt,twoside]{amsart}
\usepackage{amssymb,amsmath,amsthm, amscd, enumerate, mathrsfs}
\usepackage{graphicx, hhline}
\usepackage[all]{xy}
\usepackage{color}
\usepackage{hyperref}
\hypersetup{colorlinks=true}

\title[Existence of lc modifications and its 
applications]{Existence of log canonical modifications and its applications}
\author{Osamu Fujino and Kenta Hashizume}
\date{2022/4/28, version 0.21}
\keywords{dlt blow-ups, log canonical modifications, inversion of 
adjunction, lengths of extremal rational curves, Mori hyperbolicity, quasi-log schemes}
\address{Osamu Fujino \\ Department of Mathematics, Graduate School of Science, 
Kyoto University, Kyoto 606-8502, Japan}
\email{fujino@math.kyoto-u.ac.jp}
\address{Kenta Hashizume \\ Graduate School of 
Mathematical Sciences, 
The University of Tokyo, 3-8-1 Komaba Meguro-ku 
Tokyo 153-8914, Japan}
\email{hkenta@ms.u-tokyo.ac.jp}

\newtheorem{thm}{Theorem}[section]
\newtheorem{lem}[thm]{Lemma}
\newtheorem{cor}[thm]{Corollary}

\theoremstyle{definition}
\newtheorem{step}{Step}
\newtheorem{case}{Case}
\newtheorem{ex}[thm]{Example}
\newtheorem{defn}[thm]{Definition}
\newtheorem{rem}[thm]{Remark}
\newtheorem{nota}[thm]{Notation}
\newtheorem*{ack}{Acknowledgments}  
\makeatletter
    
    \@addtoreset{equation}{section}
\makeatother

\begin{document}

\maketitle 

\begin{abstract}
The main purpose of this paper is to establish some useful 
partial resolutions of singularities for pairs from the minimal model 
theoretic viewpoint. 
We first establish the existence of log canonical modifications 
of normal pairs under some suitable assumptions. 
Then we recover Kawakita's inversion of adjunction on log canonicity 
in full generality. We also discuss the existence of 
semi-log canonical modifications for demi-normal pairs and 
construct dlt blow-ups with several extra good properties. 
As an application, we study lengths of extremal rational curves. 
\end{abstract}

\tableofcontents

\section{Introduction}\label{x-sec1}

The main purpose of this paper is to establish some useful 
partial resolutions of singularities for pairs from the minimal model 
theoretic viewpoint. 

Let us start with an elementary example. 
Let $X$ be a normal surface. Then it is well known that 
there exists a unique minimal resolution of 
singularities $f\colon Y\to X$ of $X$. 
It plays a crucial role for the study of singularities 
of $X$. Let $g\colon Z\to X$ be any resolution of singularities of 
$X$. Then we can see $f\colon Y\to X$ as a relative minimal 
model of $Z$ over $X$. When $X$ is a higher-dimensional 
quasi-projective variety and $g\colon Z\to X$ is a resolution of 
singularities of $X$, we can always construct a relative 
minimal model $f\colon Y\to X$ of $Z$ over $X$ by running 
a minimal model program (see \cite{bchm}). Unfortunately, 
however, $Y$ may be singular. In general, $Y$ has $\mathbb Q$-factorial 
terminal singularities. Since the singularities of $Y$ is milder than 
those of $X$, $f\colon Y\to X$ sometimes plays an important 
role as a partial resolution of singularities of $X$.

In the recent study of 
higher-dimensional algebraic varieties, we know that 
it is natural and useful to 
treat pairs. 
Let us consider a quasi-projective log canonical pair $(X, \Delta)$. 
Based on \cite{bchm}, Hacon constructed a projective birational 
morphism $f\colon (Y, \Delta_Y)\to (X, \Delta)$ from a 
$\mathbb Q$-factorial divisorial log terminal pair $(Y, \Delta_Y)$ 
with $K_Y+\Delta_Y=f^*(K_X+\Delta)$ (see \cite{fujino-funda} and 
\cite{kollar-kovacs}). 
We usually call $f\colon (Y, \Delta_Y)\to (X, \Delta)$ a 
{\em{dlt blow-up}} of $(X, \Delta)$. By dlt blow-ups, many problems on 
log canonical pairs can be reduced to those on 
$\mathbb Q$-factorial divisorial log terminal pairs. 
We can see $f\colon (Y, \Delta_Y)\to (X, \Delta)$ as a 
partial resolution of singularities of $(X, \Delta)$ from the 
minimal model theoretic viewpoint. 

In this paper, we are mainly interested in pairs whose 
singularities are not necessarily log canonical. 

\subsection{Existence of log canonical modifications}
We first establish the existence of log canonical 
modifications, which is a kind of partial resolution of singularities of 
pairs from the minimal model theoretic viewpoint. 

\begin{thm}[Log canonical modifications]\label{x-thm1.1}
Let $X$ be a normal variety 
and let $\Delta$ be an effective $\mathbb R$-divisor 
on $X$ such that $K_X+\Delta$ is $\mathbb R$-Cartier. 
Let $B$ be an $\mathbb R$-divisor on $X$ 
such that the coefficients of $B$ belong to $[0,1]$, 
$\Delta-B$ is effective and ${\rm Supp} B={\rm Supp} \Delta$.
Then there exists a log canonical modification of $X$ and $B$, 
that is, a log canonical pair $(Y, B_Y)$ and 
a projective birational morphism 
$f\colon Y\to X$ such that 
\begin{itemize}
\item[\em{(i)}]
the divisor $B_{Y}$ is the sum of $f^{-1}_{*}B$ and 
the 
reduced $f$-exceptional divisor $E$, that is, $E=\sum _j E_j$ 
where $E_j$ are the $f$-exceptional prime divisors on $Y$, and
\item[\em{(ii)}]
the divisor $K_{Y}+B_{Y}$ is $f$-ample. 
\end{itemize}
\end{thm} 

Let us see an application of Theorem \ref{x-thm1.1}. 

\begin{ex}\label{x-ex1.2}
Theorem \ref{x-thm1.1} shows that every 
pair consisting of a quasi-projective 
variety $X$ and a boundary 
$\mathbb{R}$-divisor $\Delta_X$ always has an 
effective $\mathbb{R}$-divisor $D$ 
on $X$ whose coefficients are 
arbitrarily small such that there exists a 
log canonical modification of $X$ and $\Delta_X+D$. 
We note that $K_X+\Delta_X$ is not assumed to be 
$\mathbb R$-Cartier. 
Indeed, we can 
pick $A\geq 0$ so that $K_{X}+\Delta_X+A$ is an ample 
$\mathbb R$-Cartier divisor with ${\rm Supp}(\Delta_X+A)
\ne {\rm Supp \Delta_X}$. 
Then, for any $\epsilon>0$, we may find $D\geq 0$ 
such that 
${\rm Supp}(\Delta_X+D)={\rm Supp}(\Delta_X+A)$, 
$A\geq D$, the coefficients of $\Delta_X+D$ belong to 
$[0, 1]$, and 
the coefficients of $D$ 
are less than $\epsilon$. 
Then, by Theorem \ref{x-thm1.1}, we see 
that there exists 
a log canonical modification of $X$ and 
$\Delta_X+D$. 
\end{ex}

The following special case of 
Theorem \ref{x-thm1.1} is important for applications of this paper. 

\begin{thm}\label{x-thm1.3}
Let $X$ be a normal variety 
and let $\Delta$ be an effective $\mathbb R$-divisor 
on $X$ such that $K_X+\Delta$ is $\mathbb R$-Cartier. 
We put 
\begin{equation*} 
B=\Delta^{<1}+{\rm Supp}\Delta^{\geq 1}. 
\end{equation*}  
Then there exists a log canonical modification of $X$ and $B$. 
\end{thm}

Theorem \ref{x-thm1.1} is a generalization of 
\cite[Theorem 1.2]{odaka-xu}. 
Odaka and Xu proved Theorem \ref{x-thm1.1} 
under the extra assumption that $\Delta=B$ 
and $B$ is a $\mathbb{Q}$-divisor. 
By Theorem \ref{x-thm1.3}, we can recover 
Kawakita's inversion of adjunction on log canonicity 
(see Corollary 
\ref{x-cor5.5}). 

\begin{thm}[see {\cite{kawakita}} and Corollary 
\ref{x-cor5.5}]\label{x-thm1.4}
Let $(X, S+B)$ be a normal pair such that 
$S$ is a reduced divisor, $B$ is effective, and 
$S$ and $B$ have no common irreducible components. 
Let $\nu\colon S^\nu\to S$ be the normalization of $S$. 
We put $K_{S^\nu}+B_{S^\nu}=\nu^*(K_X+S+B)$. 
Then $(X, S+B)$ is log canonical 
near $S$ if and only if $(S^\nu, B_{S^\nu})$ is log 
canonical. 
\end{thm}

Kawakita's original proof of Theorem \ref{x-thm1.4} in 
\cite{kawakita} does not use the minimal model program. 
There are some attempts to recover Kawakita's 
inversion of adjunction by using the minimal model program 
under extra assumptions (see \cite{odaka-xu} and \cite{hacon}). 
We note that, in Theorem \ref{x-thm1.4}, the divisor $B$ 
is an effective $\mathbb R$-divisor which may not be 
a boundary $\mathbb R$-divisor. Hence, this is the first time 
to recover Kawakita's inversion of adjunction on log canonicity 
in full generality as an application of the minimal model program. 

For equidimensional reduced and reducible schemes, 
Koll\'ar and Shepherd-Barron constructed minimal semi-resolutions 
of surfaces (see \cite[Proposition 4.10]{ksb}). 
As a higher-dimensional generalization, 
Fujita (see \cite{fujita}) established semi-terminal modifications 
of demi-normal pairs. Here, we note that 
a {\em{demi-normal scheme}} means an 
equidimensional reduced scheme which satisfies Serre's $S_2$ 
condition and is normal crossing in codimension one. 
On the other hand, Odaka and Xu treated semi-log canonical modifications 
of demi-normal pairs in \cite[Corollary 1.2]{odaka-xu}. 
The following theorem is a generalization of \cite[Corollary 1.2]{odaka-xu} 
and Theorem \ref{x-thm1.3} for $\mathbb{Q}$-divisors. 

\begin{thm}[see Theorem \ref{x-thm4.4}]\label{x-thm1.5}
Let $X$ be a demi-normal scheme, and let $\Delta$ 
be an effective $\mathbb{Q}$-divisor on $X$ such that 
${\rm Supp}\Delta$ does not contain any codimension one 
singular loci and $K_{X}+\Delta$ is $\mathbb{Q}$-Cartier. 
We put 
\begin{equation*} 
B=\Delta^{<1}+{\rm Supp}\Delta^{\geq 1}. 
\end{equation*}  
Then $X$ equipped with $B$ has a semi-log canonical 
modification, 
that is, a semi-log canonical pair $(Y,B_{Y})$ and a projective 
birational morphism 
$f\colon Y\to X$ such that 
\begin{itemize}
\item[\em{(i)}]
$f$ is an isomorphism over 
the generic point of any codimension one singular locus,
\item[\em{(ii)}] 
$B_{Y}$ is the sum of the birational 
transform of $B$ on $Y$ and the reduced $f$-exceptional divisor, and 
\item[\em{(iii)}]
$K_{Y}+B_{Y}$ is $f$-ample.
\end{itemize}
\end{thm}

We remark that $K_X+B$ in 
Theorem \ref{x-thm1.5} is not necessarily $\mathbb{Q}$-Cartier. 
As in \cite[Example 3.1]{odaka-xu}, there is an 
example of demi-normal pairs having no semi-log 
canonical modifications. 
In our proof of Theorem \ref{x-thm1.5}, the $\mathbb R$-Cartier 
property of $K_X+\Delta$ is crucial to apply the gluing theory of Koll\'ar. 
For the details, see Remark \ref{x-rem4.5}. 

\subsection{Special crepant models}
By combining the idea of the 
proof of Theorem \ref{x-thm1.1} with 
the minimal model theory for $\mathbb Q$-factorial divisorial 
log terminal pairs, we obtain Theorem \ref{x-thm1.6}, which 
is a generalization of \cite[Lemma 3.10]{fujino}. 
Note that the morphism $g\colon (Y, \Delta_Y)\to (X, \Delta)$ in 
Theorem \ref{x-thm1.6} is a 
kind of dlt blow-up with some extra good properties. 
Here, we call it a {\em{special crepant model}} of $(X, \Delta)$. 
 
\begin{thm}[Special crepant models]\label{x-thm1.6}
Let $X$ be a normal quasi-projective variety 
and let $\Delta$ an effective $\mathbb{R}$-divisor on 
$X$ such that $K_{X}+\Delta$ 
is $\mathbb{R}$-Cartier. 
Then we can construct a crepant model 
$g \colon (Y,\Delta_{Y}) \to (X,\Delta)$, that is, a projective 
birational morphism $g\colon Y\to X$ from 
a normal $\mathbb Q$-factorial variety 
$Y$ and an effective $\mathbb{R}$-divisor 
$\Delta_{Y}$ on $Y$ such that 
\begin{itemize}
\item[\em{(i)}] $K_Y+\Delta_Y=g^*(K_X+\Delta)$, 
\item[\em{(ii)}] the pair $(Y, \Delta'_Y)$ is dlt, where 
$\Delta'_Y=\Delta^{<1}_Y+{\rm Supp} \Delta^{\geq 1}_Y$, such that 
$K_Y+\Delta'_Y$ is $g$-semi-ample, 
\item[\em{(iii)}] every $g$-exceptional prime divisor 
is a component of $(\Delta'_Y)^{=1}$, 
\item[\em{(iv)}] $g^{-1}({\rm Nklt}(X, \Delta))$ coincides with 
${\rm Nklt}(Y, \Delta_Y)$ and ${\rm Nklt}(Y, \Delta'_Y)$ set theoretically, 
\item[\em{(v)}] $g^{-1}({\rm Nlc}(X, \Delta))$ coincides with 
${\rm Nlc}(Y, \Delta_Y)$ and ${\rm Supp} \Delta^{>1}_Y$ set theoretically, and 
\item[\em{(vi)}] there is an 
effective $\mathbb R$-divisor $\Gamma_Y$ on $Y$ such that 
\begin{itemize}
\item[\em{(a)}] ${\rm Supp}\Gamma_Y=g^{-1}({\rm Nklt}(X,\Delta))=
{\rm Supp}\Delta^{\geq 1}_Y$ set theoretically, 
\item[\em{(b)}]  $-\Gamma_Y$ is $g$-semi-ample, and 
\item[\em{(c)}] $\Delta_Y-\Gamma_Y$ is effective 
and $(Y, \Delta_Y-\Gamma_Y)$ is klt.  
\end{itemize}
\end{itemize}
\end{thm}

The main difference between Theorem \ref{x-thm1.6} 
and the usual notion of dlt blow-ups is (ii). 
The $g$-semi-ampleness of $K_Y+\Delta'_Y$ is 
highly nontrivial. We note that $\Delta_Y\ne \Delta'_Y$ 
holds when $(X, \Delta)$ is not log canonical. 
 
We also note that we only need the minimal model program 
essentially obtained in 
\cite{bchm} for the proof of \cite[Lemma 3.10]{fujino}. 
On the other hand, the proof of 
Theorems \ref{x-thm1.1} and \ref{x-thm1.6} is much harder because 
it heavily depends on the minimal model theory 
for log canonical pairs discussed in \cite{has-class}. 
Although Theorem \ref{x-thm1.6} may look artificial, 
it seems to have many applications. 

\subsection{Extremal rational curves}
As an application of Theorem \ref{x-thm1.6}, we prove:

\begin{thm}\label{x-thm1.7}  
Let $X$ be a normal variety and let $\Delta$ be an effective 
$\mathbb R$-divisor on $X$ such that 
$K_X+\Delta$ is $\mathbb R$-Cartier. 
Let $\pi\colon X\to S$ be a projective morphism 
onto a scheme $S$ such that 
$-(K_X+\Delta)$ is $\pi$-ample. 
We assume that 
\begin{equation*} 
\pi\colon {\rm Nklt}(X, \Delta)\to \pi({\rm Nklt}(X, \Delta))
\end{equation*} 
is finite. Let $P$ be a closed point of $S$ 
such that there exists a curve 
$C^\dag \subset \pi^{-1}(P)$ with 
${\rm Nklt}(X, \Delta)\cap C^\dag\ne \emptyset$. 
Then there exists a non-constant morphism 
\begin{equation*} 
f\colon  \mathbb A^1\longrightarrow 
\left(X\setminus {\rm Nklt}(X, \Delta)\right)\cap \pi^{-1}(P)
\end{equation*} 
such that 
the curve $C$, the closure of $f(\mathbb A^1)$ in $X$, 
is a {\em{(}}possibly singular{\em{)}} rational curve 
satisfying  
$C\cap {\rm Nklt}(X, \Delta)\ne \emptyset$ with 
\begin{equation*} 
0<-(K_X+\Delta)\cdot C\leq 1. 
\end{equation*} 
\end{thm}

Theorem \ref{x-thm1.7} is a kind of generalization of 
\cite[Theorem 1.8]{fujino}. 
We note that a log canonical pair, any union of some log canonical 
centers of a log canonical pair, and a quasi-projective 
semi-log canonical pair have natural quasi-log scheme 
structures. Therefore, the theory of quasi-log schemes 
can be seen as a framework to treat all the above 
objects on an equal footing. For the details of the 
theory of quasi-log schemes, see \cite[Chapter 6]{fujino-foundations} 
and \cite{fujino}. 
We will quickly explain the basic definitions in Section \ref{x-sec7}. 
By combining Theorem \ref{x-thm1.7} with 
the framework of quasi-log schemes 
discussed in \cite{fujino}, we have: 

\begin{thm}\label{x-thm1.8}
Let $[X, \omega]$ be a quasi-log scheme 
and let $\pi\colon  X\to S$ be a projective 
morphism between schemes such that 
$-\omega$ is $\pi$-ample 
and that 
\begin{equation*} 
\pi\colon {\rm Nqklt} (X, \omega)\to \pi({\rm Nqklt}(X, \omega))
\end{equation*} 
is finite. 
Let $P$ be a closed point of $S$ such that 
there exists a curve $C^\dag\subset \pi^{-1}(P)$ with 
${\rm Nqklt}(X, \omega)\cap C^\dag\ne \emptyset$. 
Then there exists a non-constant morphism 
\begin{equation*} 
f\colon \mathbb A^1\longrightarrow \left(X\setminus {\rm Nqklt}(X, \omega)\right)
\cap \pi^{-1}(P) 
\end{equation*} 
such that $C$, the closure of $f(\mathbb A^1)$ in $X$, satisfies 
$C\cap {\rm Nqklt}(X, \omega)\ne \emptyset$ with 
\begin{equation*} 
0<-\omega\cdot C\leq 1. 
\end{equation*} 
\end{thm}

Theorem \ref{x-thm1.8} completely solves the first author's 
conjecture (see \cite[Conjecture 1.15]{fujino}). 
As an easy direct consequence of Theorem \ref{x-thm1.8}, 
we establish: 

\begin{thm}[Lengths of extremal rational curves for quasi-log schemes]
\label{x-thm1.9}
Let $[X, \omega]$ be a quasi-log scheme and let $\pi\colon X\to S$ 
be a projective 
morphism between schemes. 
Let $R_j$ be an $\omega$-negative extremal ray of $\overline {NE}(X/S)$ 
that are rational and relatively ample at infinity and 
let $\varphi_{R_j}$ be the contraction morphism associated to $R_j$. 
Let $U_j$ be any open qlc 
stratum of $[X, \omega]$ such that 
$\varphi_{R_j}\colon \overline {U_j}\to \varphi_{R_j}(\overline {U_j})$ 
is not finite and that 
$\varphi_{R_j}\colon W^\dag\to \varphi_{R_j}
(W^\dag)$ is finite for every qlc center 
$W^\dag$ of $[X, \omega]$ with $W^\dag 
\subsetneq \overline {U_j}$, where 
$\overline {U_j}$ is the closure of $U_j$ in $X$. 
Let $P$ be a closed point of $\varphi_{R_j}(U_j)$. 
If there exists a curve $C^\dag$ such that 
$\varphi_{R_j}(C^\dag)=P$, 
$C^\dag\not \subset U_j$,  
and $C^\dag\subset \overline {U_j}$, 
then there exists a non-constant 
morphism 
\begin{equation*} 
f_j\colon \mathbb A^1\longrightarrow U_j\cap \varphi^{-1}_{R_j}(P)
\end{equation*} 
such that 
$C_j$, the closure of $f_j(\mathbb A^1)$ in $X$, 
spans $R_j$ in $N_1(X/S)$ and satisfies 
$C_j\not\subset U_j$ with  
\begin{equation*} 
0<-\omega\cdot C_j\leq 1.
\end{equation*} 
\end{thm}

Note that Theorem \ref{x-thm1.9} supplements \cite[Theorem 1.6 (iii)]{fujino}. 
We also note that 
the above results generalize \cite[Theorem 3.1]{lz} completely. 
The following example may help the reader understand 
Theorem \ref{x-thm1.9}. 

\begin{ex}\label{x-ex1.10}
This example shows that the condition $C^\dag\not \subset U_j$ 
is necessary for the estimate of 
the length of $C_{j}$ in Theorem \ref{x-thm1.9}. 
Let $H_1, \ldots, H_n$ be general hyperplanes on $X=\mathbb P^n$. 
We put $\Delta=\sum _{i=1}^n H_i$ and $\Delta'=\sum _{i=1}^{n-1}H_i$. 
Let us consider the structure morphism $\pi\colon X
\to S={\rm Spec} (\mathbb C)$. 
We note that $(X, \Delta)$ and 
$(X, \Delta')$ are log canonical and 
that $-(K_X+\Delta)$ and $-(K_X+\Delta')$ are $\pi$-ample. 
Since the Picard number of $X$ is one, 
$\pi\colon X\to S$ is an extremal contraction. 
Let $C$ be any one-dimensional lc center of $(X, \Delta)$. 
Then it is easy to see that $C\simeq \mathbb P^1$, 
$-(K_X+\Delta)\cdot C=1$, and 
the open lc center
associated to $C$ is isomorphic to $\mathbb A^1$. 
On the other hand, there are no zero-dimensional 
lc centers of $(X, \Delta')$ and 
$-(K_X+\Delta')\cdot C'\geq 2$ holds for every curve $C'$ on $X$. 
\end{ex}

We summarize the contents of this paper. 
In Section \ref{x-sec2}, we collect some basic definitions and 
results for the reader's convenience. 
Section \ref{x-sec3} is the main part of this paper. 
We prove Theorems \ref{x-thm1.1}, \ref{x-thm1.3}, and 
\ref{x-thm1.6} by using the minimal model theory for log 
canonical pairs. The main ingredient of this section 
is the second author's 
theorem:~Theorem \ref{x-thm3.1}, which was obtained in 
\cite{has-class}. 
In Section \ref{x-sec4}, we discuss 
semi-log canonical modifications for demi-normal pairs. 
We prove Theorem \ref{x-thm1.5} by using Theorem \ref{x-thm1.3} and 
Koll\'ar's gluing theory in \cite{kollar-mmp}. 
The readers who are interested only 
in normal pairs can skip this section. 
In Section \ref{x-sec5}, we treat inversion of adjunction on 
log canonicity. We first prove a slight generalization of 
Hacon's inversion of adjunction on log canonicity for 
log canonical centers. Then we recover 
Kawakita's inversion of adjunction in full 
generality (see Theorem \ref{x-thm1.4}) as 
a special case. 
Section \ref{x-sec6} is devoted to the proof of 
Theorem \ref{x-thm1.7}, which heavily depends on 
the minimal model program for normal pairs. 
In Section \ref{x-sec7}, 
we quickly review some basic definitions in the theory 
of quasi-log schemes. In Section \ref{x-sec8}, 
we prove Theorems \ref{x-thm1.8} and 
\ref{x-thm1.9} by using Theorem \ref{x-thm1.7} and 
the framework of quasi-log schemes. 
We note that we need quasi-log schemes only in Sections \ref{x-sec7} and 
\ref{x-sec8}.

\begin{ack}
The authors thank Christopher Hacon very much 
for answering their question. 
They also thank the referee for many useful comments and 
suggestions. 
\end{ack}

We will work over $\mathbb C$, the 
complex number field, throughout this paper. 
In this paper, a {\em{scheme}} means a 
separated scheme of finite type over $\mathbb C$. 
A {\em{variety}} means an integral scheme, that is, an 
irreducible and reduced separated scheme of finite type over 
$\mathbb C$.

\section{Preliminaries}\label{x-sec2}

In this paper, we use the theory of minimal models 
for higher-dimensional log canonical pairs. 
Here we collect some definitions and results for the 
reader's convenience. 
For the details, see \cite{fujino-funda}, \cite{fujino-foundations}, 
\cite{kollar-mmp}, and \cite{kollar-mori}.  

\begin{defn}[Singularities of pairs]\label{x-def2.1}
Let $X$ be a variety and let $E$ be a prime divisor on $Y$ 
for some birational
morphism $f\colon Y\to X$ from a normal variety $Y$. 
Then $E$ is called a divisor {\em{over}} $X$. 
A {\em{normal pair}} $(X, \Delta)$ consists of 
a normal variety $X$ and an $\mathbb R$-divisor on $X$ 
such that $K_X+\Delta$ is $\mathbb R$-Cartier. 
Let $(X, \Delta)$ be a normal pair and let 
$f\colon Y\to X$ be a projective 
birational morphism from a normal variety $Y$. 
Then we can write 
\begin{equation*} 
K_Y=f^*(K_X+\Delta)+\sum _E a(E, X, \Delta)E
\end{equation*} 
with 
\begin{equation*} 
f_*\left(\underset{E}\sum a(E, X, \Delta)E\right)=-\Delta, 
\end{equation*} 
where $E$ runs over prime divisors on $Y$. 
We call $a(E, X, \Delta)$ the {\em{discrepancy}} of $E$ with 
respect to $(X, \Delta)$. 
Note that we can define the discrepancy $a(E, X, \Delta)$ for 
any prime divisor $E$ over $X$ by taking a suitable 
resolution of singularities of $X$. 
If $a(E, X, \Delta)\geq -1$ (resp.~$>-1$) for 
every prime divisor $E$ over $X$, 
then $(X, \Delta)$ is called {\em{sub log canonical}} (resp.~{\em{sub 
kawamata log terminal}}). 
We further assume that $\Delta$ is effective. 
Then $(X, \Delta)$ is 
called {\em{log canonical}} 
({\em{lc}}, for short) 
and {\em{kawamata log terminal}} 
({\em{klt}}, for short) 
if it is sub log canonical and sub kawamata log terminal, respectively. 

Let $(X, \Delta)$ be a log canonical pair. If there 
exists a projective birational morphism 
$f\colon Y\to X$ from a smooth variety $Y$ such that 
both ${\rm Exc}(f)$, 
the exceptional locus of $f$,  
and  ${\rm Exc}(f)\cup {\rm Supp} f^{-1}_*\Delta$ are simple 
normal crossing divisors on $Y$ and that 
$a(E, X, \Delta)>-1$ holds for every $f$-exceptional divisor $E$ on $Y$, 
then $(X, \Delta)$ is called {\em{divisorial log terminal}} ({\em{dlt}}, for short). 
\end{defn}

\begin{defn}[Non-klt loci, non-lc loci, and lc centers]\label{x-def2.2} 
Let $(X,\Delta)$ be a normal pair. 
If there exist a projective birational morphism 
$f\colon Y\to X$ from a normal variety $Y$ and a prime divisor $E$ on $Y$ 
such that $(X, \Delta)$ is 
sub log canonical in a neighborhood of the 
generic point of $f(E)$ and that 
$a(E, X, \Delta)=-1$, then $f(E)$ is called a 
{\em{log canonical center}} (an {\em{lc center}}, for short) 
of 
$(X, \Delta)$. 

From now on, we further assume that 
$\Delta$ is effective. 
Then the {\em{non-klt locus}} of $(X,\Delta)$, 
denoted by ${\rm Nklt}(X, \Delta)$, is the smallest 
closed subset $Z$ of $X$ whose complement 
$(X\setminus Z, \Delta|_{X\setminus Z})$ is a klt pair. 
Similarly, the {\em non-lc locus} of $(X,\Delta)$, 
denoted by ${\rm Nlc}(X, \Delta)$, is the smallest 
closed subset $Z$ of $X$ such that the 
complement $(X\setminus Z, \Delta|_{X\setminus Z})$ is log canonical. 
\end{defn}

\begin{defn}\label{x-def2.3}
Let $X$ be an equidimensional reduced scheme and 
let $D=\sum _i d_i D_i$ be an $\mathbb R$-divisor 
on $X$ such that $d_i$ is a real number and 
$D_i$ is an 
irreducible reduced closed subscheme of $X$ of pure codimension one 
for every $i$ with $D_i\ne D_j$. 
We put 
\begin{alignat*}{2}D^{<1}&=\sum _{d_i<1}d_i D_i,&  
&D^{\leq 1}= \sum _{d_i\leq 1} d_iD_i, \qquad 
D^{= 1}= \sum _{d_i= 1} D_i, 
\\D^{\geq 1}&= \sum _{d_i\geq 1} d_iD_i,&\qquad &{\text{and}} 
\qquad D^{> 1}= \sum _{d_i> 1} d_iD_i. \end{alignat*}
We also put 
\begin{equation*} 
\lfloor D\rfloor =\sum _i \lfloor d_i 
\rfloor D_i, \quad 
\lceil D\rceil =-\lfloor -D\rfloor, 
\quad {\text{and}} \quad \{D\}=D-\lfloor D\rfloor,  
\end{equation*} 
where $\lfloor d_i \rfloor$ is the integer defined by $d_i-1<\lfloor 
d_i \rfloor \leq d_i$. 
We say that $D$ is a {\em{boundary divisor}} 
if $D$ is effective and $D=D^{\leq 1}$. We say 
that $D$ is a {\em{reduced divisor}} if $D=D^{=1}$.  
\end{defn}

\begin{nota}\label{x-nota2.4}
Let $f\colon X\dashrightarrow X'$ be a birational 
map of normal 
varieties and let $D$ be an $\mathbb{R}$-divisor on $X$. 
If there is no risk of confusion, $D_{X'}$ denotes the sum of $f_{*}D$ 
and the reduced $f^{-1}$-exceptional divisor $E$ on $X'$, that is, 
$E=\sum_j E_{j}$ where $E_{j}$ are the 
$f^{-1}$-exceptional prime divisors on $X'$. 
\end{nota}

\begin{defn}\label{x-def2.5}
Let $p\colon V\to W$ be a projective surjective morphism 
from a normal variety $V$ to a variety $W$ and 
let $D_1$ and $D_2$ be $\mathbb R$-Cartier divisors on $V$. 
Then $D_1\sim _{\mathbb R, W}D_2$ means 
that there exists an $\mathbb R$-Cartier divisor 
$D$ on $W$ such that $D_1-D_2\sim _{\mathbb R} p^*D$. 
We say that $D_1$ is $\mathbb R$-linearly equivalent 
to $D_2$ over $W$ when $D_1\sim _{\mathbb R, W}D_2$. 
\end{defn}

In this paper, we adopt the following definition of 
{\em{models}}. 

\begin{defn}[Models]\label{x-def2.6}
Let $(X,\Delta)$ be a log canonical pair and $X\to Z$ 
a projective morphism to a variety $Z$. 
Let $X'\to Z$ be a projective morphism from a 
normal variety and let $\phi\colon X\dashrightarrow X'$ 
be a birational map over $Z$. 
Let $E$ be the reduced $\phi^{-1}$-exceptional 
divisor on $X'$, that is, $E=\sum_j E_{j}$ where $E_{j}$ are the 
$\phi^{-1}$-exceptional prime divisors on $X'$. 
Put $\Delta'=\phi_{*}\Delta+E$. 
If $K_{X'}+\Delta'$ is $\mathbb{R}$-Cartier, then 
the pair $(X', \Delta')$ is called 
a {\em{log birational model}} of $(X,\Delta)$ over $Z$. 
A log birational model $(X', \Delta')$ of $(X,\Delta)$ over $Z$ is called 
a {\em{good minimal model}} if 
\begin{itemize}
\item[(i)]
$X'$ is $\mathbb{Q}$-factorial, 
\item[(ii)]
$K_{X'}+\Delta'$ is semi-ample over $Z$, and 
\item[(iii)]
for any prime divisor $D$ on $X$ which is exceptional over $X'$, we have
\begin{equation*} 
a(D, X, \Delta) < a(D, X', \Delta'). 
\end{equation*} 
\end{itemize}
A log birational model $(X',\Delta')$ of 
$(X, \Delta)$ over $Z$ is called a {\em{Mori fiber space}} 
if $X'$ is $\mathbb Q$-factorial and there is a 
contraction $X' \to W$ over $Z$ with $\dim W<\dim X'$ such that 
\begin{itemize}
\item[(iv)]
the relative Picard number $\rho(X'/W)$ is 
one and $-(K_{X'}+\Delta')$ is ample over $W$, and 
\item[(v)]
for any prime divisor $D$ over $X$, we have
\begin{equation*} 
a(D,X,\Delta)\leq a(D,X',\Delta')
\end{equation*} 
and strict inequality holds if $D$ is a divisor on $X$ and 
is exceptional over $X'$.
\end{itemize}
\end{defn} 

We make two important remarks on the 
minimal model program for log canonical pairs.

\begin{rem}\label{x-rem2.7}
Let $(X,\Delta)$ be a $\mathbb Q$-factorial dlt pair and 
$\pi \colon X \to Z$ a projective morphism 
from a normal quasi-projective variety $X$ to a quasi-projective variety $Z$. 
If $(X,\Delta)$ has a good minimal model 
or a Mori fiber space over $Z$ as in Definition \ref{x-def2.6}, 
then all $(K_{X}+\Delta)$-minimal model 
programs over $Z$ with scaling of an ample divisor terminate (see 
\cite[Theorem 4.1]{birkar-flip}). 
\end{rem}

\begin{rem}\label{x-rem2.8}
Let $\pi \colon X \to Z$ be a projective 
morphism from a normal quasi-projective 
variety $X$ to a quasi-projective variety $Z$. 
Let $(X,\Delta)$ and $(X,\Delta')$ be two 
$\mathbb Q$-factorial dlt pairs such that 
$K_{X}+\Delta'\sim_{\mathbb R, Z}t(K_X+\Delta)$ 
for a positive real number $t$. 
Suppose that $(X,\Delta)$ has a good minimal model over $Z$. 
By Remark \ref{x-rem2.7}, there exists 
a $(K_X+\Delta)$-minimal model program over 
$Z$ with scaling of an ample divisor that terminates after finitely 
many steps. 
Because any $(K_X+\Delta)$-minimal 
model program over $Z$ with scaling 
of an ample divisor is also a $(K_X+\Delta')$-minimal 
model program over $Z$ with scaling of an ample divisor, 
we see that there is a $(K_X+\Delta')$-minimal model 
program over $Z$ terminating with a good minimal model. 
Thus, we see that $(X,\Delta')$ has a good minimal model over $Z$. 
\end{rem}

\begin{defn}[Log canonical modifications]\label{x-def2.9}
Let $X$ be a normal variety and let $B$ be 
a boundary $\mathbb R$-divisor on $X$. 
Then, a {\em log canonical modification of $X$ 
and $B$} is a log canonical pair $(Y, B_Y)$ and a 
projective birational morphism 
$f\colon Y\to X$ such that 
\begin{itemize}
\item[(i)]
the divisor $B_{Y}$ is the sum of $f^{-1}_{*}B$ 
and the reduced $f$-exceptional divisor $E$, 
that is, $E=\sum _j E_j$ where $E_j$ are 
the $f$-exceptional prime divisors on $Y$, and
\item[(ii)]
the divisor $K_Y+B_Y$ is $f$-ample. 
\end{itemize}
In this paper, if there is no risk of confusion, then 
the notation $f\colon (Y, B_Y)\to (X, B)$ 
denotes the structure of a log canonical 
modification when there is a log canonical modification of $X$ and $B$. 
\end{defn}

In this paper, we will freely 
use the existence of {\em dlt blow-ups}, which was obtained in \cite{fujino}. 
Note that a dlt blow-up is sometimes called a {\em{dlt modification}} in 
the literature. 

\begin{thm}[{Dlt blow-ups, see \cite[Theorem 3.9]{fujino}}]\label{x-thm2.10}
Let $X$ be a normal quasi-projective variety and let $\Delta=\sum _i 
d_i \Delta_i$ be an effective $\mathbb R$-divisor on $X$ such 
that $K_X+\Delta$ is $\mathbb R$-Cartier. 
In this case, we can construct a projective birational morphism 
$f\colon Y\to X$ from a normal quasi-projective variety $Y$ with the 
following properties. 
\begin{itemize}
\item[\em{(i)}] $Y$ is $\mathbb Q$-factorial. 
\item[\em{(ii)}] $a(E, X, \Delta)\leq -1$ for every $f$-exceptional 
divisor $E$ on $Y$. 
\item[\em{(iii)}] We put 
\begin{equation*}
\Delta^\dag:=\sum _{0<d_i< 1}d_i f^{-1}_*\Delta_i+\sum _{d_i\geq 1} 
f^{-1}_*\Delta_i+\sum _{\text{$E$:~$f$-exceptional}} E.  
\end{equation*}
Then $(Y, \Delta^\dag)$ is dlt and the following 
equality 
\begin{equation*} 
K_Y+\Delta^\dag=f^*(K_X+\Delta)+\sum _{a(E, X, \Delta)<-1}
(a(E, X, \Delta)+1)E
\end{equation*} 
holds. 
\end{itemize}
\end{thm}

Note that $\Delta$ is not necessarily a boundary divisor in 
Theorem \ref{x-thm2.10}. We close this section with an important 
remark on Theorem \ref{x-thm2.10}. 

\begin{rem}\label{x-rem2.11} 
Let us recall how to construct $f\colon Y\to X$ in 
Theorem \ref{x-thm2.10}. 
In the proof of \cite[Theorem 3.9]{fujino}, 
we first take a suitable resolution of singularities of $X$ and 
then run a minimal model program over $X$. 
After finitely many flips and divisorial contractions over $X$, 
we get a desired birational map $f\colon Y\to X$. 
Hence we may further assume that $f\colon 
Y\to X$ is the identity map over some nonempty Zariski open subset 
of $X$ in Theorem \ref{x-thm2.10}. 
More precisely, in Theorem \ref{x-thm2.10}, 
let $U$ be the largest Zariski open subset 
of $X$ such that $(U, \Delta|_U)$ has only $\mathbb Q$-factorial kawamata 
log terminal singularities. Then we can make $f$ the identity map over $U$.
\end{rem}

\section{Proof of Theorems \ref{x-thm1.1}, \ref{x-thm1.3}, 
and \ref{x-thm1.6}}\label{x-sec3}

In this section, we prove Theorems \ref{x-thm1.1}, 
\ref{x-thm1.3}, and \ref{x-thm1.6}. 
One of the main ingredients of this section  
is the second author's following result on the minimal model program 
for log canonical pairs. 

\begin{thm}[{\cite[Corollary 3.6]{has-class}}] \label{x-thm3.1}
Let $\pi\colon X\to Z$ be a projective morphism of 
normal quasi-projective varieties and let $(X, B)$ be a log canonical pair. 
Suppose that there is an effective $\mathbb{R}$-divisor $D$ on $X$ 
such that
\begin{itemize}
\item[\em{(a)}] 
$-(K_{X}+B+D)$ is nef over $Z$, and 
\item[\em{(b)}]  
$(X,B+a D)$ is log canonical for some positive real number $a$.  
\end{itemize}
Then, $(X,B)$ has a good minimal model or a Mori fiber space over $Z$.
\end{thm}

Before we prove Theorem \ref{x-thm1.1}, we 
prepare an elementary lemma. 
 
\begin{lem}\label{x-lem3.2}
Let $X$ be a normal variety and $B$ a boundary $\mathbb{R}$-divisor on $X$. 
Suppose that there are two log canonical 
modifications $f\colon (Y,B_Y)\to (X, B)$ 
and $f'\colon (Y',B_{Y'})\to (X, B)$ of $X$ and $B$. 
Then the induced birational map 
$\phi:=f'^{-1}\circ f \colon Y\dashrightarrow Y'$ 
is an isomorphism and $\phi_{*}B_Y=B_{Y'}$. 
\end{lem}

\begin{proof}
Let $h\colon W\to Y$ and $h'\colon W\to Y'$ 
be a common resolution of $\phi$. 
We define an $\mathbb{R}$-divisor $E$ on $W$ by 
\begin{equation*} 
E:=h^{*}(K_Y+B_Y)-h'^{*}(K_{Y'}+B_{Y'}). 
\end{equation*} 
Since $\phi$ is a birational map over $X$, 
every component $D$ of $E$ is exceptional over $X$. 
If a component $D$ of $E$ is not $h$-exceptional, 
then $h_{*}D$ is exceptional over $X$.   
Thus we have $a(D,Y,B_{Y})=-{\rm coeff}_{h_{*}D}(B_{Y})=-1$. 
On the other hand, we have $a(D,Y',B_{Y'})\geq -1$ 
because $(Y',B_{Y'})$ is log canonical. 
So, we obtain ${\rm coeff}_{D}(E)\geq 0$. 
Applying the negativity 
lemma (\cite[Lemma 3.6.2 (2)]{bchm}) to $h\colon W\to X$ 
and $E$, we have $E\geq0$. 
We apply the same argument to $-E$, then we obtain $-E \geq0$. 
Therefore, it follows that $E=0$. 
Since $K_{Y}+B_{Y}$ and $K_{Y'}+B_{Y'}$ are both ample over $X$, 
$\phi$ is an isomorphism and $\phi_{*}B_{Y}=B_{Y'}$. 
\end{proof}

\begin{rem}\label{x-rem3.3}
Let $X$ be a smooth projective variety and let $g\colon X\to X$ 
be any automorphism of $X$. 
Then $g\colon X\to X$ is a log canonical modification of $X$ and 
$B=0$ by definition. 
 
In some geometric applications, we implicitly require 
that a log canonical modification $f\colon (Y, B_Y)\to 
(X, B)$ satisfies the extra assumption that 
$f$ is the identity morphism 
over some nonempty Zariski open subset of $X$. 
Under this extra assumption, by Lemma \ref{x-lem3.2}, 
the log canonical modification $f\colon (Y, B_Y)\to (X, B)$ 
of $X$ and $B$ is unique if it exists. 
\end{rem}
Let us prove Theorem \ref{x-thm1.1}. 

\begin{proof}[Proof of Theorem \ref{x-thm1.1}]
In Step \ref{x-step1.1.1}, we will prove Theorem \ref{x-thm1.1} 
under the extra assumption that $X$ is quasi-projective. 
Then, in Step \ref{x-step1.1.2}, we will treat the general case. 
\setcounter{step}{0}
\begin{step}\label{x-step1.1.1} 
In this step, we will prove Theorem \ref{x-thm1.1} 
under the extra assumption that $X$ is quasi-projective. 
Hence, from now on, we assume that $X$ is quasi-projective. 

We take a dlt blow-up $g \colon Z \to X$ with $K_Z+\Delta_Z=g^*(K_X+\Delta)$ 
as in Theorem \ref{x-thm2.10}, that is, 
$g$ is a projective birational morphism such that 
every $g$-exceptional prime divisor $F$ satisfies $a(F,X,\Delta) \leq -1$ and 
that $(Z, \Delta^{<1}_{Z}+
{\rm Supp}\Delta^{\geq 1}_{Z})$ is a $\mathbb Q$-factorial dlt pair. 
Note that we may further assume that $g$ is the identity morphism 
over some nonempty Zariski open subset of $X$ by Remark 
\ref{x-rem2.11}. 

We define an $\mathbb R$-divisor 
$B_{Z}$ on $Z$ to be the sum of $g^{-1}_{*}B$ 
and the reduced $g$-exceptional divisor (Notation \ref{x-nota2.4}). 
Then the relations 
\begin{equation*} 
B_{Z}\geq0 \qquad {\text{and}} 
\qquad \bigl(\Delta^{<1}_{Z}+{\rm Supp}\Delta^{\geq 1}_{Z}\bigr)-B_{Z}\geq0
\end{equation*}  
hold since the coefficients of $B$ 
belong to $[0,1]$ and $\Delta-B$ is effective. 
This implies that the pair $(Z, B_{Z})$ is a $\mathbb Q$-factorial dlt pair. 
We will prove that $(Z,B_{Z})$ 
has a good minimal model over $X$. 
We put 
\begin{equation*} 
D_{Z}=\Delta_{Z}-B_{Z}. 
\end{equation*} 
We have $\Delta_{Z}-\bigl(\Delta^{<1}_{Z}
+{\rm Supp}\Delta^{\geq 1}_{Z}\bigr)\geq0$ by construction, so
\begin{equation*} 
D_{Z}=\Delta_{Z}-B_{Z}\geq \bigl(\Delta^{<1}_{Z}
+{\rm Supp}\Delta^{\geq 1}_{Z}\bigr)-B_{Z}\geq0, 
\end{equation*} 
from which $D_{Z}$ is an effective $\mathbb{R}$-divisor on $Z$. 
Furthermore, recalling ${\rm Supp} B={\rm Supp}\Delta$ 
and that $B_{Z}$ is the sum of $g^{-1}_{*}B$ 
and the reduced $g$-exceptional 
divisor, it follows that ${\rm Supp}\Delta_{Z}={\rm Supp} B_{Z}$. 
Thus, we see that 
\begin{equation*}
\begin{split}
{\rm Supp} D_{Z}\subset {\rm Supp}\Delta_{Z}={\rm Supp} B_{Z}. 
\end{split}
\end{equation*}
We can find a real number $t>0$ 
such that $B_{Z}-tD_{Z}\geq0$. 
Then the pair $(Z,B_{Z}-tD_{Z})$ is dlt 
because $(Z, B_{Z})$ is a dlt pair and $D_{Z}$ is effective. 
Since $K_{Z}+\Delta_{Z}=g^{*}(K_{X}+\Delta)$, we have 
\begin{equation*} 
K_{Z}+B_{Z}=K_{Z}+\Delta_{Z}-D_{Z} \sim_{\mathbb{R},X}-D_{Z}. 
\end{equation*} 
By this relation, we obtain 
\begin{equation*} 
K_{Z}+B_{Z} -tD_{Z}\sim_{\mathbb{R},X}-(1+t)
D_{Z}\sim_{\mathbb{R},X}(1+t)(K_{Z}+B_{Z}). 
\end{equation*} 
By Remark \ref{x-rem2.8}, the existence of a good minimal 
model of $(Z,B_{Z})$ over $X$ follows from the 
existence of a good minimal model of 
$(Z,B_{Z}-tD_{Z})$ over $X$. 
We put 
\begin{equation*} 
\tilde{B}_{Z}=B_{Z}-tD_{Z}. 
\end{equation*} 
Then $K_{Z}+\tilde{B}_{Z}+(1+t)D_{Z}\sim_{\mathbb{R},X}0$ 
and $(Z,\tilde{B}_{Z}+tD_{Z})$ is dlt since $\tilde{B}_{Z}+tD_Z=B_Z$ 
by definition. 
By Theorem \ref{x-thm3.1}, $(Z,\tilde{B}_{Z})$ has a 
good minimal model over $X$. 
Therefore, $(Z,B_{Z})$ 
also has a good minimal model over $X$. 

By running a minimal model program over $X$, 
we get a good minimal model $(Z', B_{Z'})$ of $(Z, B_Z)$ over $X$ 
(see Remark \ref{x-rem2.7}). 
Let $Z' \to Y$ be the contraction over $X$ induced by $K_{Z'}+B_{Z'}$. 
We define $B_{Y}$ to be the birational transform 
of $B_{Z'}$ on $Y$. Then it is easy to check that $(Y,B_{Y})$ 
is a log canonical pair and 
the induced morphism $f\colon Y\to X$ is the 
desired birational morphism. 
By construction, we may assume that $f\colon Y\to X$ is 
the identity morphism over some nonempty Zarsiki open subset 
of $X$. 
\end{step}
\begin{step}\label{x-step1.1.2} 
In this step, we will treat the general case, that is, 
$X$ is not necessarily quasi-projective. 

We take a finite affine open covering $X=\bigcup_{i}U_{i}$. 
By Step \ref{x-step1.1.1}, there exist log canonical modifications 
$f_i\colon (V_i, B_{V_i}) \to (U_i, B|_{U_i})$ of $U_i$ and $B|_{U_i}$ 
such that 
$f_i$ is the identity morphism 
over some nonempty Zariski open subset of $U_i$ for 
all $i$. By Lemma \ref{x-lem3.2} (see also 
Remark \ref{x-rem3.3}), 
$f_i\colon (V_i, B_{V_i}) \to (U_i, B|_{U_i})$ coincides with 
$f_j\colon (V_j, B_{V_j}) \to (U_j, B|_{U_j})$ over $U_i\cap 
U_j$ for every $j\ne i$. 
Therefore, we get a log canonical modification of $X$ and $B$ by gluing 
them. 
\end{step}
We finish the proof of Theorem \ref{x-thm1.1}. 
\end{proof}

\begin{proof}[Proof of Theorem \ref{x-thm1.3}] 
It is a special case of Theorem \ref{x-thm1.1}. 
\end{proof}

The following remark easily follows from the 
definition of log canonical modifications. 
It is very useful for various geometric applications. 

\begin{rem}\label{x-rem3.4}
Let $X$ be a normal variety 
and let $\Delta$ be an effective $\mathbb R$-divisor 
on $X$ such that $K_X+\Delta$ is $\mathbb R$-Cartier. 
We put $B=\Delta^{<1}+{\rm Supp}\Delta^{\geq 1}.$ 
Let $f\colon (Y, B_Y)\to (X, B)$ be a log canonical 
modification of $X$ and $B$. 
We give two important remarks.

We put $U=X\setminus f({\rm Exc}(f))$. 
Then, any point $x$ of $X$ is contained in $U$ if and 
only if $K_{X}+B$ is $\mathbb{R}$-Cartier and $(X,B)$ is log canonical near $x$. 

We define $\Delta_{Y}$ by $K_{Y}+\Delta_{Y}=f^{*}(K_{X}+\Delta)$, 
and we put $\Gamma_{Y}=\Delta_{Y}-B_{Y}$. 
Then, it follows that $\Gamma_{Y}$ is effective, $-\Gamma_{Y}$ is 
ample over $X$, and we have ${\rm Exc}(f)\subset {\rm Supp}\Gamma_{Y}=
{\rm Supp}\Delta^{>1}_{Y}$. 
\end{rem}

By the same argument as in the proof of Theorem \ref{x-thm1.1}, 
we obtain:

\begin{lem}[Good dlt blow-ups]\label{x-lem3.5}
Let $X$ be a normal quasi-projective variety and let $\Delta=\sum _i 
d_i \Delta_i$ be an effective $\mathbb R$-divisor on $X$ such 
that $K_X+\Delta$ is $\mathbb R$-Cartier. 
Then there exists a projective birational morphism 
$f\colon Y\to X$ as in Theorem \ref{x-thm2.10} 
such that $K_Y+\Delta^\dag$ in 
Theorem \ref{x-thm2.10} is $f$-semi-ample. 
\end{lem}
\begin{proof}
We put $K_Y+\Delta_Y=f^*(K_X+\Delta)$. 
Then 
\begin{equation*} 
\Delta^\dag=\Delta^{<1}_Y+{\rm Supp} \Delta^{\geq 1}_Y
\end{equation*} 
holds. Therefore, as in the proof of 
Theorem \ref{x-thm1.1} (put $B=\Delta^{<1}+{\rm Supp}\Delta^{\geq 1}$ 
in the proof of Theorem \ref{x-thm1.1}), 
by Theorem \ref{x-thm3.1}, 
the dlt pair 
\begin{equation*} 
(Y, \Delta^\dag)=
(Y, \Delta^{<1}_Y+{\rm Supp} \Delta^{\geq 1}_Y)
\end{equation*}  
has a good minimal model over $X$. 
Hence, after finitely many flips and divisorial contractions, 
we can make $K_Y+\Delta^\dag$ $f$-semi-ample. 
\end{proof}

\begin{rem}\label{x-rem3.6}
As in Remark \ref{x-rem2.11}, by construction, 
we may further assume that $f$ is the identity morphism 
over some nonempty Zariski open subset 
of $X$ in Theorems \ref{x-thm1.1}, \ref{x-thm1.3}, 
and Lemma \ref{x-lem3.5}. 
\end{rem}

Let us prove Theorem \ref{x-thm1.6}. 

\begin{proof}[Proof of Theorem \ref{x-thm1.6}]
Let 
$f \colon (Z,\Delta_{Z})\to (X,\Delta)$ 
be a good dlt blow-up as in Lemma \ref{x-lem3.5}. This 
means that $f\colon Z\to X$ is a projective birational morphism 
from a normal $\mathbb Q$-factorial variety $Z$ satisfying 
(i)--(iii). If necessary, then we may further assume that 
$f$ is the identity morphism over some 
nonempty Zariski open subset of $X$ (see Remark \ref{x-rem3.6}). 

\setcounter{step}{0}
\begin{step}\label{x-step1.4.1}
In this step, we will run a $(K_Z+\Delta^{<1}_{Z}+(1-\epsilon)
{\rm Supp}\Delta^{\geq 1}_{Z})$-minimal model program over $X$ and 
make $K_Z+\Delta^{<1}_{Z}+(1-\epsilon)
{\rm Supp}\Delta^{\geq 1}_{Z}$ semi-ample over $X$ 
for some small number $\varepsilon$. 

We can always take $\epsilon>0$ such that, 
for any $(K_{Z}+\Delta^{<1}_{Z}+(1-\epsilon)
{\rm Supp}\Delta^{\geq 1}_{Z})$-minimal 
model program over $X$, the divisor 
$K_{Z}+\Delta^{<1}_{Z}+{\rm Supp}\Delta^{\geq 1}_{Z}$ 
is numerically 
trivial over all extremal contractions of the steps of the minimal model 
program. 
This fact follows from the well-known estimate of 
lengths of extremal rational curves 
(see, for example, \cite[Proposition 3.2]{birkar-existII}). 
More explicitly, if $\varepsilon$ is sufficiently small, 
then we can use \cite[Proposition 3.2 (4) and (5)]{birkar-existII} 
to prove that $K_Z+\Delta^{<1}_Z+{\rm Supp}\Delta^{\geq 1}_Z$ 
is numerically trivial on each step of any 
$(K_Z+\Delta^{<1}_Z+(1-\varepsilon) {\rm Supp}\Delta^{\geq 1}_Z)$-minimal 
model program since 
$K_Z+\Delta^{<1}_Z+{\rm Supp}\Delta^{\geq 1}_Z$ is nef over $X$. 
Since $(Z,\Delta^{<1}_{Z}+(1-\epsilon){\rm Supp}\Delta^{\geq 1}_{Z} )$ is klt, 
we can run a 
$(K_{Z}+\Delta^{<1}_{Z}+(1-\epsilon){\rm Supp}\Delta^{\geq 1}_{Z})$-minimal 
model program 
over $X$ and finally obtain a good minimal model 
$(Z',\Delta^{<1}_{Z'}+(1-\epsilon){\rm Supp}\Delta^{\geq 1}_{Z'} )$ 
over $X$ by \cite{bchm}. 
By the choice of $\epsilon$, the divisor 
$K_{Z'}+\Delta^{<1}_{Z'}+{\rm Supp}\Delta^{\geq 1}_{Z'}$ is semi-ample over $X$. 

Therefore, for any $u \in [0,\epsilon]$, the divisor
\begin{equation*} 
K_{Z'}+\Delta^{<1}_{Z'}+(1-u){\rm Supp}\Delta^{\geq 1}_{Z'}
\end{equation*} 
is semi-ample 
over $X$. 
Note that the divisor 
$-(\Delta^{\geq 1}_{Z'}-(1-\epsilon){\rm Supp}\Delta^{\geq 1}_{Z'})$ is 
also semi-ample over $X$ because
\begin{equation*} 
K_{Z'}+\Delta^{<1}_{Z'}+
(1-\epsilon){\rm Supp}\Delta^{\geq 1}_{Z'}
\sim_{\mathbb{R},X}-(\Delta^{\geq 1}_{Z'}-(1-\epsilon){\rm Supp} 
\Delta^{\geq 1}_{Z'})
\end{equation*}  
holds. 
By the above construction of $(Z', \Delta_{Z'})$, the pair
$(Z', \Delta^{< 1}_{Z'}+{\rm Supp} \Delta^{\geq 1}_{Z'})$ 
is lc and 
\begin{equation*} 
{\rm Nklt}\left(Z', \Delta^{< 1}_{Z'}+{\rm Supp} \Delta^{\geq 1}_{Z'}
\right)={\rm Supp} \Delta^{\geq 1}_{Z'}
\end{equation*} 
holds set theoretically. 
\end{step}

\begin{step}\label{x-step1.4.2}
The morphism $Z'\to X$ is denoted by $\alpha'$. Then we take a 
dlt blow-up $\beta \colon Y \to Z'$ of 
$(Z', \Delta^{< 1}_{Z'}+{\rm Supp}\Delta^{\geq 1}_{Z'})$ 
such that 
$a(E, Z', \Delta^{< 1}_{Z'}+{\rm Supp}\Delta^{\geq 1}_{Z'})=-1$ 
holds for every $\beta$-exceptional divisor $E$ on $Y$ 
(see Theorem \ref{x-thm2.10}). 
We set $g= \alpha' \circ \beta$ and define 
$\Delta_{Y}$ by $K_{Y}+\Delta_{Y}=\beta^{*}\alpha'^{*}(K_{X}+\Delta)$. 
By the definition of $\Delta_Y$, $K_Y+\Delta_Y=g^*(K_X+\Delta)$ 
obviously holds. This means that $g\colon (Y, \Delta_Y)\to 
(X, \Delta)$ satisfies (i). 
By the construction of $\alpha'\colon Z'\to X$, 
$a(E, X, \Delta)\leq -1$ holds for every 
$\alpha'$-exceptional divisor $E$ on $Z'$. 
By the construction of $\beta\colon Y\to Z'$, 
$a(E, X, \Delta)=a(E, Z', \Delta_{Z'})\leq a(E, Z', \Delta^{<1}_{Z'}
+{\rm Supp} \Delta^{\geq 1}_{Z'})=-1$ holds 
for every $\beta$-exceptional divisor $E$ on $Y$. 
This means that every $g$-exceptional prime divisor is a 
component of $(\Delta'_Y)^{=1}$. 
Therefore, $g\colon (Y, \Delta_Y)\to (X, \Delta)$ satisfies 
(iii). 
By the construction of the dlt blow-up 
$\beta\colon Y\to Z'$ of $(Z', \Delta^{<1}_{Z'}+{\rm Supp} \Delta^{\geq 1}_{Z'})$, 
\begin{equation*}
K_Y+\Delta'_Y=\beta^*(K_{Z'}+
\Delta^{<1}_{Z'}+{\rm Supp} \Delta^{\geq 1}_{Z'})
\end{equation*} 
holds. Hence, $K_Y+\Delta'_Y$ is semi-ample over $X$. 
Thus, $g\colon (Y, \Delta_Y)\to (X, \Delta)$ satisfies (ii). 

From now on, we will explain how to check (iv). 
We put 
\begin{equation*} 
E_{Z'}=
\Delta^{\geq 1}_{Z'}-(1-\epsilon){\rm Supp}\Delta^{\geq 1}_{Z'}. 
\end{equation*} 
Then  ${\rm Supp} E_{Z'}={\rm Supp}\Delta^{\geq 1}_{Z'}$, 
and $-E_{Z'}$ is semi-ample 
over $Z$ by Step \ref{x-step1.4.1}. 
We can also see that 
\begin{equation*} 
{\rm Supp}\beta^{*}E_{Z'}={\rm Supp}
\Delta^{\geq 1}_{Y}={\rm Nklt}(Y,\Delta_{Y})
\end{equation*} 
holds 
and $-\beta^{*}E_{Z'}$ is semi-ample 
over $X$. 
Now $g^{-1}({\rm Nklt}(X,\Delta))\supset {\rm Supp}\Delta^{\geq 1}_{Y}$ is 
obvious  
and it is easy to check that $g({\rm Supp}\Delta^{\geq 1}_{Y})=
{\rm Nklt}(X,\Delta)$ holds set theoretically. 
If $g^{-1}({\rm Nklt}(X,\Delta))\supsetneq 
{\rm Supp}\Delta^{\geq 1}_{Y}$, then there is a curve 
$C\subset Y$ such that $g(C)\in {\rm Nklt}(X,\Delta)$ 
and $(C\cdot \beta^{*}E_{Z'})>0$ 
since $g({\rm Supp}\Delta^{\geq 1}_Y)={\rm Nklt}(X,\Delta)$ and 
$g$ has connected fibers.
This is a contradiction because $-\beta^{*}E_{Z'}$ is nef 
over $X$. 
Hence we see that $g^{-1}({\rm Nklt}(X,\Delta))
={\rm Supp}\Delta^{\geq 1}_{Y}$ holds. 
This means that $g \colon (Y, \Delta_Y)\to (X, \Delta)$ satisfies 
(iv). 

For (v), we note that $\Delta_Y-\Delta'_Y$ is effective 
and that $-(\Delta_Y-\Delta'_Y)
\sim _{\mathbb R, X}K_Y+\Delta'_Y$ is $g$-semi-ample. 
By the definition of $\Delta'_Y$, ${\rm Supp} (\Delta_Y-\Delta'_Y)={\rm Supp} 
\Delta^{>1}_Y$ holds. 
By the same argument as in the proof of (iv) above, 
we can check that 
\begin{equation*} 
g^{-1}({\rm Nlc}(X, \Delta))={\rm Supp} \Delta^{>1}_Y={\rm Nlc}(Y, \Delta_Y)
\end{equation*} 
holds set theoretically. 
 
Finally, we will construct $\Gamma_{Y}$ as in (vi). 
Since the pair $(Z',\Delta^{<1}_{Z'}+
(1-u){\rm Supp}\Delta^{\geq 1}_{Z'})$ is klt and 
$K_{Z'}+\Delta^{<1}_{Z'}+(1-u){\rm Supp}\Delta^{\geq 1}_{Z'}$
is semi-ample over $X$ for every $u\in(0,\epsilon]$, by 
the construction of $\beta \colon Y\to Z'$, we can find a 
positive real number $u$ such that if we set $\Delta^{u}_{Y}$ by 
\begin{equation*} 
K_{Y}+\Delta^{u}_{Y}=\beta^{*}(K_{Z'}+
\Delta^{<1}_{Z'}+(1-u){\rm Supp}\Delta^{\geq 1}_{Z'}), 
\end{equation*} 
then $\Delta^{u}_{Y}$ is effective, 
$(Y,\Delta^{u}_{Y})$ is klt, and $K_{Y}+\Delta^{u}_{Y}$ 
is semi-ample over $X$. 
Fix such $u>0$ and put
\begin{equation*} 
\Gamma_{Y}:=\Delta_{Y}-\Delta^{u}_{Y}
=\beta^{*}\left(\Delta_{Z'}-(\Delta^{<1}_{Z'}+
(1-u){\rm Supp}\Delta^{\geq 1}_{Z'})\right). 
\end{equation*} 
Note that $\Gamma_{Y}=(K_{Y}+\Delta_{Y})
-(K_{Y}+\Delta^{u}_{Y})\sim_{\mathbb{R}, X}
-(K_{Y}+\Delta^{u}_{Y})$ holds. Hence $-\Gamma_{Y}$ 
is semi-ample over $X$. 
It is clear that $(Y,\Delta_{Y}-\Gamma_{Y})$ 
is klt because $\Delta_{Y}-\Gamma_{Y}=\Delta^{u}_{Y}$. 
Since 
\begin{equation*} 
{\rm Supp}\left(\Delta_{Z'}-(\Delta^{<1}_{Z'}
+(1-u){\rm Supp}\Delta^{\geq 1}_{Z'})\right)=
{\rm Supp}\Delta^{\geq 1}_{Z'}={\rm Supp} E_{Z'}, 
\end{equation*} 
we have ${\rm Supp}\Gamma_{Y}=
{\rm Supp}\beta^{*}E_{Z'}$, thus
\begin{equation*} 
{\rm Supp}\Gamma_{Y}={\rm Supp}
\Delta^{\geq 1}_{Y}=g^{-1}({\rm Nklt}(X,\Delta))
\end{equation*}  
holds. 
In this way, we see that $\Gamma_{Y}$ 
satisfies all the desired conditions in (vi). 
\end{step}
We complete the proof of Theorem \ref{x-thm1.6}. 
\end{proof}

\begin{rem}\label{x-rem3.7} 
By construction (see also 
Remarks \ref{x-rem2.11} and 
\ref{x-rem3.6}), we may further assume that 
$g\colon (Y, \Delta_Y)\to (X, \Delta)$ in 
Theorem \ref{x-thm1.6} is the identity morphism 
over some nonempty Zariski open subset of $X$. 
Hence we can see $g\colon (Y, \Delta_Y)\to (X, \Delta)$ 
as a partial resolution of 
singularities of the pair $(X, \Delta)$. 
More precisely, in Theorem \ref{x-thm1.6}, 
let $U$ be the largest Zariski open subset 
of $X$ such that $(U, \Delta|_U)$ has only $\mathbb Q$-factorial kawamata 
log terminal singularities. Then we can make $g$ the identity map over $U$.
\end{rem}

\section{On semi-log canonical modifications of 
demi-normal pairs}\label{x-sec4}

A {\em demi-normal scheme} $X$ is a 
reduced and equidimensional scheme which satisfies Serre's $S_2$ 
condition and is normal crossing in codimension one. 
For basic definitions and properties of demi-normal pairs and 
semi-log canonical pairs, 
see \cite[Sections 5.1 and 5.2]{kollar-mmp}. 
In this section, we prove the existence 
of semi-log canonical modifications of demi-normal 
pairs (see Theorem \ref{x-thm4.4}). 
Let us start with the following lemma. 

\begin{lem}\label{x-lem4.1}
Let $(X, S+B)$ be a log canonical pair such that 
$B$ is an effective $\mathbb{R}$-divisor 
and $S$ is a prime divisor with the normalization $S^{\nu}$. 
Let $\Gamma$ be an effective $\mathbb{R}$-Cartier divisor on $X$ 
such that ${\rm Supp}\Gamma\subset \lfloor B\rfloor$. 
We define an effective $\mathbb{R}$-divisor 
$B_{S^{\nu}}$ on $S^{\nu}$ by applying adjunction to $(X,S+B)$ and $S^{\nu}$. 
We put $\Gamma_{S^{\nu}}$ as the pullback of $\Gamma$ to $S^{\nu}$. 
If $\Gamma_{S^{\nu}}\neq 0$, then for any 
component $P_{S^{\nu}}$ of $\Gamma_{S^{\nu}}$ 
we have ${\rm coeff}_{P_{S^{\nu}}}(B_{S^{\nu}})=1$. 
In particular, if ${\rm Supp}\Gamma$ intersects 
$S$, then the pair $(S^{\nu}, B_{S^{\nu}}+\Gamma_{S^{\nu}})$ 
is not log canonical.  
\end{lem}

\begin{proof}
Note that $\Gamma_{S^{\nu}}$ is well-defined 
as an effective $\mathbb{R}$-Cartier divisor on $S^{\nu}$ because $S$ is not a 
component of $\lfloor B\rfloor$ and 
${\rm Supp}\Gamma \subset \lfloor B\rfloor$. 
Since the problem is local, by shrinking $X$, 
we may assume that $X$ is quasi-projective. 

Let $P_{S}$ be the image of $P_{S^{\nu}}$ on $X$. 
Then $P_{S}$ is a subvariety of $X$ of codimension two and 
$P_{S}\subset S\cap {\rm Supp}\Gamma$. 
We take a dlt blow-up $f\colon (Y,T+B_{Y})\to (X,S+B)$ (see 
Theorem \ref{x-thm2.10}), where $T=f^{-1}_{*}S$ and $B_{Y}$ is 
the sum of $f^{-1}_{*}B$ and the reduced $f$-exceptional divisor. 
We put $\Gamma_{Y}=f^{*}\Gamma$. 
Note that $T$ is not a component of $\Gamma_{Y}$. 

The facts $P_{S}\subset S\cap {\rm Supp}\Gamma$, 
$f(T)=S$, and ${\rm Supp}\Gamma_{Y}=f^{-1}({\rm Supp}\Gamma)$ show the 
inclusion $P_{S}\subset f(T\cap {\rm Supp}
\Gamma_{Y})$. 
Because $P_{S}$ and all irreducible components 
of $T\cap {\rm Supp}\Gamma_{Y}$ have the same 
dimension, we can find an irreducible component $D_{T}$ of $T\cap 
{\rm Supp}\Gamma_{Y}$ such that $f(D_{T})=P_{S}$. 
Furthermore, every component 
of $\Gamma_{Y}$ is a component of $\lfloor B_{Y}\rfloor$. 
It is because ${\rm Supp}\Gamma\subset\lfloor B\rfloor$ 
and all $f$-exceptional prime divisors on $Y$ are components 
of $\lfloor B_{Y}\rfloor$. 
Thus, it follows that $D_{T}$ is an irreducible 
component of $T\cap \lfloor B_{Y}\rfloor$. 
We define $B_T$ to be the $\mathbb R$-divisor 
on $T$ with $K_T+B_T=(K_Y+T+B_Y)|_T$. 
Since $(Y, T+B_Y)$ is a $\mathbb{Q}$-factorial 
dlt pair, we have ${\rm coeff}_{D_{T}}(B_{T})=1$. 
Since $f(D_T)=P_S$ and 
$B_{S^{\nu}}$ is the birational transform 
of $B_{T}$ on $S^{\nu}$, we obtain ${\rm coeff}_{P_{S^{\nu}}}(B_{S^{\nu}})=1$. 

If ${\rm Supp}\Gamma$ intersects 
$S$, then $\Gamma_{S^{\nu}}\neq 0$ and any irreducible component 
$P_{S^{\nu}}$ of $\Gamma_{S^{\nu}}$ satisfies 
\begin{equation*} 
{\rm coeff}_{P_{S^{\nu}}}(B_{S^{\nu}}
+\Gamma_{S^{\nu}})>{\rm coeff}_{P_{S^{\nu}}}(B_{S^{\nu}})=1 
\end{equation*} 
by the above discussion. 
Therefore, the pair $(S^{\nu}, B_{S^{\nu}}
+\Gamma_{S^{\nu}})$ is not log canonical.  
\end{proof}

\begin{lem}\label{x-lem4.2}
Let $X$ be a normal variety and 
let $\Delta$ be an effective $\mathbb{R}$-divisor 
on $X$ such that $K_{X}+\Delta$ is $\mathbb{R}$-Cartier. 
Let $S$ be a component of $\Delta^{=1}$ with 
the normalization $S^{\nu}$. 
We put $B=\Delta^{<1}+{\rm Supp}\Delta^{\geq 1}$, 
and let $f\colon (Y,B_{Y})\to (X,B)$ be a log canonical 
modification of $X$ and $B$. 
We put $T=f^{-1}_{*}S$ with the 
normalization $T^{\nu}$, and 
let $\bar{f}\colon T^{\nu}\to S^{\nu}$ be the 
birational morphism induced by $f$. 
\begin{equation*} 
\xymatrix
{
 T^{\nu}\ar[d]_{\bar{f}}\ar[r]&T\ar[d]\ar@{}[r]|*{\subset}&Y \ar[d]^{f}
\\
S^{\nu}\ar[r]&S\ar@{}[r]|*{\subset} &X
}
\end{equation*} 
Let $\Delta_{S^{\nu}}$ be 
the effective $\mathbb{R}$-divisor 
on $S^{\nu}$ defined 
by applying adjunction to $(X,\Delta)$ and $S^{\nu}$, and let $B_{T^{\nu}}$ 
be the effective $\mathbb{R}$-divisor on $T^{\nu}$ defined 
by applying adjunction to $(Y,B_{Y})$ and $T^{\nu}$. 
We put $B_{S^{\nu}}=
\Delta^{<1}_{S^{\nu}}+{\rm Supp}\Delta^{\geq 1}_{S^{\nu}}$. 

Then, the relation $B_{S^{\nu}}=\bar{f}_{*}B_{T^{\nu}}$ 
holds and 
$\bar{f} \colon (T^{\nu},B_{T^{\nu}})\to (S^{\nu},B_{S^{\nu}})$ is 
a log canonical modification of $S^{\nu}$ and 
$B_{S^{\nu}}$. 
\end{lem}

\begin{proof}
As in Step \ref{x-step1.1.2} in the proof of Theorem \ref{x-thm1.1}, 
it is sufficient to take a finite affine open covering $X=\bigcup _i U_i$ and 
prove this lemma on each open subset $U_i$. 
Therefore, 
we may assume that $X$ is quasi-projective. 
It is obvious from construction that the 
pair $(T^{\nu},B_{T^{\nu}})$ is log canonical 
and $K_{T^{\nu}}+B_{T^{\nu}}$ is ample over $S^{\nu}$. 
Thus, it is enough to prove that 
$B_{T^{\nu}}$ is the sum of $\bar{f}^{-1}_{*}B_{S^{\nu}}$ 
and the reduced $\bar{f}$-exceptional divisor. 

We define $\Delta_{Y}$ by $K_{Y}+\Delta_{Y}
=f^{*}(K_{X}+\Delta)$, and we define 
$\Gamma_{Y}$ by $\Gamma_{Y}=\Delta_{Y}-B_{Y}$. 
By Remark \ref{x-rem3.4}, $\Gamma_{Y}$ is effective and 
${\rm Supp}\Gamma_{Y}={\rm Supp}\Delta^{>1}_{Y}\supset {\rm Exc}(f)$. 
Let $\Delta_{T^{\nu}}$ be the 
effective divisor on $T^{\nu}$ defined 
by applying adjunction to $(Y,\Delta_{Y})$ and $T^{\nu}$, and 
let $\Gamma_{T^{\nu}}$ be the pullback of $\Gamma_{Y}$ to $T^{\nu}$. 
Then $\Delta_{T^{\nu}}=B_{T^{\nu}}+
\Gamma_{T^{\nu}}$ and $\bar{f}_{*}\Delta_{T^{\nu}}=\Delta_{S^{\nu}}$. 

We pick a prime divisor $P_{T^{\nu}}$ on $T^{\nu}$. 
If $\Gamma_{T^{\nu}}\neq 0$ and $P_{T^{\nu}}$ 
is a component of $\Gamma_{T^{\nu}}$, then we 
can apply Lemma \ref{x-lem4.1} 
to $(Y, B_{Y})$, $T$ and $\Gamma_{Y}$. 
In this way, we obtain ${\rm coeff}_{P_{T^{\nu}}}(B_{T^{\nu}})=1$ and 
${\rm coeff}_{P_{T^{\nu}}}(\Delta_{T^{\nu}})>1$. 
On the other hand, if $P_{T^{\nu}}$ is not 
a component of $\Gamma_{T^{\nu}}$, then 
we have ${\rm coeff}_{P_{T^{\nu}}}(\Delta_{T^{\nu}})
={\rm coeff}_{P_{T^{\nu}}}(B_{T^{\nu}})\leq 1$.  
From these discussions, we obtain
\begin{equation*} 
B_{T^{\nu}}=\Delta^{\leq1}_{T^{\nu}}+
{\rm Supp}\Delta^{>1}_{T^{\nu}}=
\Delta^{<1}_{T^{\nu}}+{\rm Supp}\Delta^{\geq1}_{T^{\nu}}. 
\end{equation*} 
Since $B_{S^{\nu}}=\Delta^{<1}_{S^{\nu}}+
{\rm Supp}\Delta^{\geq 1}_{S^{\nu}}$ which is the 
hypothesis of Lemma \ref{x-lem4.2}, we obtain
\begin{equation*} 
\bar{f}_{*}B_{T^{\nu}}
=\bar{f}_{*}(\Delta^{<1}_{T^{\nu}}+{\rm Supp}\Delta^{\geq1}_{T^{\nu}})
=B_{S^{\nu}}. 
\end{equation*} 
Furthermore, since ${\rm Supp}\Delta^{>1}_{Y}\supset {\rm Exc}(f)$ by 
Remark \ref{x-rem3.4}, every $\bar{f}$-exceptional 
prime divisor $E$ on $T^{\nu}$ is a 
component of $\Gamma_{T^{\nu}}$, hence ${\rm coeff}_{E}(B_{T^{\nu}})=1$. 
Therefore, we see that $B_{T^{\nu}}$ 
is the sum of $\bar{f}^{-1}_{*}B_{S^{\nu}}$ 
and the reduced $\bar{f}$-exceptional divisor. 
It follows that 
$\bar{f} \colon (T^{\nu},B_{T^{\nu}})\to (S^{\nu},B_{S^{\nu}})$ is 
a log canonical modification of 
$S^{\nu}$ and 
$B_{S^{\nu}}=\Delta^{<1}_{S^{\nu}}+{\rm Supp}\Delta^{\geq 1}_{S^{\nu}}$.
\end{proof}

\begin{lem}\label{x-lem4.3}
Let $X$ be a normal variety and $B$ a 
boundary $\mathbb{R}$-divisor on $X$. 
Suppose that there is a log canonical 
modification $f\colon (Y,B_Y)\to (X, B)$ 
and an involution $\tau$ of $X$, that is, 
$\tau$ is an isomorphism 
of $X$ such that $\tau^{2}$ is the identity morphism, 
such that $\tau_*B=B$. 
Then $\tau$ lifts to an involution $\tau'$ of $Y$ such 
that $\tau'_{*}B_{Y}=B_{Y}$. 
\end{lem}

\begin{proof}
The morphism 
$\tau\circ f \colon (Y, B_Y)\to (X, B)$ is also a 
log canonical modification of $X$ and $B$. 
Therefore, the induced birational 
map $f^{-1}\circ \tau \circ f\colon Y \to Y$ is an isomorphism 
by Lemma \ref{x-lem3.2}. 
We put $\tau'= f^{-1}\circ \tau \circ f$. 
By Lemma \ref{x-lem3.2} again, we have 
$\tau'_{*}B_{Y}=B_{Y}$. 
Moreover, the isomorphism $\tau'^{2}\colon Y\to Y$ is the identity on 
an open subset of $Y$, so $\tau'^{2}$ is the identity. 
In this way, $\tau'$ is an involution of 
$Y$ such that $\tau'_{*}B_{Y}=B_{Y}$ and $f\circ \tau'=\tau\circ f$. 
\end{proof}

The following theorem is the main result of this section. 

\begin{thm}\label{x-thm4.4}
Let $X$ be a demi-normal scheme, 
and let $\Delta$ be an effective $\mathbb{Q}$-divisor 
on $X$ such that ${\rm Supp}\Delta$ 
does not contain any codimension one 
singular loci and $K_X+\Delta$ is $\mathbb Q$-Cartier. 
We put 
\begin{equation*} 
B=\Delta^{<1}+{\rm Supp}\Delta^{\geq 1}. 
\end{equation*} 
Then $X$ equipped with $B$ has a semi-log canonical modification, 
that is, a semi-log canonical pair 
$(Y, B_{Y})$ and a projective birational morphism 
$f\colon Y\to X$ such that 
\begin{itemize}
\item[\em{(i)}]
$f$ is an isomorphism over the 
generic point of any codimension one singular locus,
\item[\em{(ii)}] 
$B_Y$ is the sum of the 
birational transform of $B$ on $Y$ and the reduced 
$f$-exceptional divisor, and 
\item[\em{(iii)}]
$K_Y+B_Y$ is $f$-ample.
\end{itemize}
\end{thm}

\begin{proof}
We follow the proof of \cite[Corollary 1.2]{odaka-xu}. 
Let $\nu \colon \bar{X}\to X$ be the normalization. 
We may write $K_{\bar{X}}+\bar{D}+
\bar{\Delta}=\nu^{*}(K_{X}+\Delta)$, 
where $\bar{D}$ is the conductor and 
$\bar{\Delta}$ is an effective $\mathbb{Q}$-divisor. 
We decompose $\bar{X}=\amalg_{i}\bar{X}_{i}$ into 
irreducible components, and we set 
$\bar{D}_{i}=\bar{D}|_{\bar{X}_{i}}$ and 
$\bar{\Delta}_{i}=\bar{\Delta}|_{\bar{X}_{i}}$. 
Then  $K_{\bar{X}_{i}}+\bar{D}_{i}+\bar{\Delta}_{i}$ 
is $\mathbb{Q}$-Cartier. 
We put 
\begin{equation*} 
\bar{B}_{i}=\bar{\Delta}^{<1}_{i}+{\rm Supp}\bar{\Delta}^{\geq 1}_{i}. 
\end{equation*}  
Then there exists a log canonical 
modification $g_{i}\colon (Y_{i},T_{Y_{i}}+B_{Y_{i}})
\to (\bar{X}_{i}, \bar{D}_{i}+\bar{B}_{i})$ of $\bar{X}_{i}$ 
and $\bar{D}_{i}+\bar{B}_{i}$, where $T_{Y_{i}}=g^{-1}_{i*}\bar{D}_{i}$ 
and $B_{Y_{i}}$ is the sum of $g^{-1}_{i*}\bar{B}_{i}$ 
and the reduced $g_{i}$-exceptional divisor. 

Fix an index $i$. 
We pick an irreducible component 
$\bar{D}_{i,j}$ of $\bar{D}_{i}$. 
Let $\bar{D}^{\nu}_{i,j}$ and 
$\bar{D}^{\nu}_{i}$ be the normalizations 
of $\bar{D}_{i,j}$ and $\bar{D}_{i}$, respectively. 
Then $\bar{D}^{\nu}_{i,j}$ is an irreducible component of $\bar{D}^{\nu}_{i}$. 
We put $T_{i,j}=g^{-1}_{i*}\bar{D}_{i,j}$ 
and let $T^{\nu}_{i,j}$ be the normalization of $T_{i,j}$.
\begin{equation*} 
\xymatrix
{
 T^{\nu}_{i,j}\ar[d]\ar[r]&T_{i,j}\ar[d]\ar@{}[r]|*{\subset}&Y_{i} \ar[d]^{g_{i}}
\\
\bar{D}^{\nu}_{i,j}\ar[r]&\bar{D}_{i,j}\ar@{}[r]|*{\subset} &\bar{X}_{i}
}
\end{equation*} 
We define $\Delta_{\bar{D}^{\nu}_{i,j}}$, 
$B_{T^{\nu}_{i,j}}$, and $B_{\bar{D}^{\nu}_{i,j}}$ as follows:
\begin{itemize}
\item[(a)] $\Delta_{\bar{D}^{\nu}_{i,j}}$ is the 
$\mathbb{Q}$-divisor on $\bar{D}^{\nu}_{i,j}$ defined 
by adjunction for $(\bar{X}_{i},\bar{D}_{i}+\bar{\Delta}_{i})$ 
and $\bar{D}^{\nu}_{i,j}$, 
\item[(b)]
 $B_{T^{\nu}_{i,j}}$ is the $\mathbb{Q}$-divisor 
 on $T^{\nu}_{i,j}$ defined 
 by adjunction for $(Y_{i},T_{Y_{i}}+B_{Y_{i}})$ and $T^{\nu}_{i,j}$, and
\item[(c)]  $B_{\bar{D}^{\nu}_{i,j}}$ is a $\mathbb{Q}$-divisor 
on $\bar{D}^{\nu}_{i,j}$ defined by 
$B_{\bar{D}^{\nu}_{i,j}}=\Delta^{<1}_{\bar{D}^{\nu}_{i,j}}+{\rm Supp}
\Delta^{\geq 1}_{\bar{D}^{\nu}_{i,j}}$. 
\end{itemize}
By Lemma \ref{x-lem4.2}, the 
morphism $( T^{\nu}_{i,j},B_{T^{\nu}_{i,j}})\to (\bar{D}^{\nu}_{i,j}, 
B_{\bar{D}^{\nu}_{i,j}})$ is a log canonical modification 
of $\bar{D}^{\nu}_{i,j}$ and $B_{\bar{D}^{\nu}_{i,j}}$. 

We freely use the notations in the previous paragraph. 
Recall that $\bar{D}$ is the conductor. 
Let $\bar{D}^{\nu}$ be the normalization of $\bar{D}$. 
By construction, we have 
\begin{equation*} 
\bar{D}^{\nu}=\amalg_{i,j}\bar{D}^{\nu}_{i,j}. 
\end{equation*} 
The construction of $\Delta_{\bar{D}^{\nu}_{i,j}}$ 
shows that $\sum_{i,j}\Delta_{\bar{D}^{\nu}_{i,j}}$ is 
the effective $\mathbb{Q}$-divisor on $\bar{D}^{\nu}$ 
defined 
by adjunction (\cite[Definition 4.2]{kollar-mmp}) 
for $(\bar{X},\bar{D}+\bar{\Delta})$ and $\bar{D}^{\nu}$. 
Since $\bar{D}$ is the conductor 
of $X$, the normalization $\bar{D}^{\nu}$ has an involution 
$\tau \colon \bar{D}^{\nu}\to\bar{D}^{\nu}$ (see \cite[5.2]{kollar-mmp}). 
Furthermore, \cite[Proposition 5.12]{kollar-mmp} 
shows that the relation 
$\tau_{*}\Bigl(\sum_{i,j}\Delta_{\bar{D}^{\nu}_{i,j}}\Bigr)
=\sum_{i,j}\Delta_{\bar{D}^{\nu}_{i,j}}$ holds. 
Since we have defined 
$B_{\bar{D}^{\nu}_{i,j}}=\Delta^{<1}_{\bar{D}^{\nu}_{i,j}}+
{\rm Supp}\Delta^{\geq 1}_{\bar{D}^{\nu}_{i,j}}$, 
we see that $\sum_{i,j}B_{\bar{D}^{\nu}_{i,j}}$ is an 
effective $\mathbb{Q}$-divisor on 
$\bar{D}^{\nu}=\amalg_{i,j}\bar{D}^{\nu}_{i,j}$ and 
\begin{equation*} 
\tau_{*}\Bigl(\sum_{i,j}B_{\bar{D}^{\nu}_{i,j}}\Bigr)
=\sum_{i,j}B_{\bar{D}^{\nu}_{i,j}}. 
\end{equation*} 
Thus, $\tau$ is an involution of $\bar{D}^{\nu}
=\amalg_{i,j}\bar{D}^{\nu}_{i,j}$ such that 
$\tau_{*}\Bigl(\sum_{i,j}B_{\bar{D}^{\nu}_{i,j}}\Bigr)
=\sum_{i,j}B_{\bar{D}^{\nu}_{i,j}}.$
Since $( T^{\nu}_{i,j},B_{T^{\nu}_{i,j}})\to 
(\bar{D}^{\nu}_{i,j}, B_{\bar{D}^{\nu}_{i,j}})$ is a log 
canonical modification, $\tau$ lifts 
to an involution $\tau'$ of $\amalg_{i,j}T^{\nu}_{i,j}$ 
such that $\tau'_{*}\Bigl(\sum_{i,j}B_{T^{\nu}_{i,j}}\Bigr)=
\sum_{i,j}B_{T^{\nu}_{i,j}}$ by Lemma \ref{x-lem4.3}. 

We have constructed the following objects 
over the normalization $\bar{X}=\amalg_{i}\bar{X}_{i}$ of $X$.
\begin{itemize}
\item[(d)]
a log canonical pair $(Y_{i},T_{Y_{i}}+B_{Y_{i}})$ and 
a projective birational morphism $Y_{i}\to \bar{X}_{i}$ 
such that $K_{Y_{i}}+T_{Y_{i}}+
B_{Y_{i}}$ is ample over $\bar{X}_{i}$, and
\item[(e)]
an involution $\tau'$ of $\amalg_{i,j}T^{\nu}_{i,j}$ 
such that 
$\tau'_{*}\Bigl(\sum_{i,j}B_{T^{\nu}_{i,j}}\Bigr)
=\sum_{i,j}B_{T^{\nu}_{i,j}}$, 
where $\amalg_{i,j}T^{\nu}_{i,j}$ is the 
normalization of $\amalg_{i}T_{Y_{i}}$. 
\end{itemize}
Using the gluing theory by Koll\'ar (\cite[Corollary 5.37, 
Corollary 5.33, and Theorem 5.38]{kollar-mmp}), we get a 
semi-log canonical pair $(Y,B_{Y})$ over $X$ whose 
normalization is $\amalg_{i}(Y_{i},T_{Y_{i}}+B_{Y_{i}})$ 
and the conductor is $\sum_{i}T_{Y_{i}}$. 
More precisely, 
by \cite[Corollary 5.37]{kollar-mmp} we see 
that the set theoretical equivalence relation 
(see \cite[Definition 9.1]{kollar-mmp}) defined with $\tau'$ is finite. 
Then, by \cite[Corollary 5.33]{kollar-mmp} we 
get a demi-normal pair $(Y,B_{Y})$ over $X$ 
whose normalization is $\amalg_{i}(Y_{i},T_{Y_{i}}+B_{Y_{i}})$ 
and the conductor is $\sum_{i}T_{Y_{i}}$. 
Finally, by \cite[Theorem 5.38]{kollar-mmp} we see 
that $(Y,B_{Y})$ is a semi-log canonical pair. 
By the construction of $(Y,B_{Y})$, 
the morphism $(Y,B_{Y})\to X$ satisfies 
all the conditions of Theorem \ref{x-thm4.4}. 
Indeed, the first condition of Theorem \ref{x-thm4.4} follows from that 
the involution $\tau'$ of $\amalg_{i,j}T^{\nu}_{i,j}$ is 
the lift of the involution $\tau$ of  $\bar{D}^{\nu}$, 
the normalization of the conductor of $X$. 
The second condition of 
Theorem \ref{x-thm4.4} follows from 
the definition of $B_{Y_{i}}$ (see the first 
paragraph of this proof), and the third 
condition of Theorem \ref{x-thm4.4} 
is obvious because $K_{Y_{i}}+T_{Y_{i}}+B_{Y_{i}}$ 
is ample over $\bar{X}_{i}$ and $\bar{X}_{i}\to X$ is a finite morphism. 
In this way, we can get a semi-log canonical modification of $X$ and $B$. 
\end{proof}

We close this section with a remark. 

\begin{rem}\label{x-rem4.5}
A key ingredient for applying the gluing 
theory of Koll\'ar is Lemma \ref{x-lem4.2}, 
which says that constructing a log canonical modification is 
compatible with adjunction. 
In Lemma \ref{x-lem4.2}, the $\mathbb{R}$-Cartier 
property of $K_{X}+\Delta$ is 
crucial for the proof. 
Therefore, the hypothesis of 
Theorem \ref{x-thm4.4} that $K_{X}+\Delta$ 
is $\mathbb{Q}$-Cartier is necessary for the proof. 
In general, as shown 
in \cite[Example 3.1]{odaka-xu}, there is a demi-normal 
scheme $X$ having no semi-log canonical modification. 
If we take the normalization $\bar{X}$ and the 
conductor $\bar{D}$ of the demi-normal 
scheme $X$ in \cite[Example 3.1]{odaka-xu}, 
then the divisor $K_{\bar{X}}+\bar{D}$ 
is not $\mathbb{Q}$-Cartier and $K_{\bar{X}}+a \bar{D}$ is 
not $\mathbb{R}$-Cartier for any $a>1$. 
See \cite[Example 3.1]{odaka-xu} for details.  
\end{rem}

\section{On inversion of adjunction on log canonicity}\label{x-sec5}

In this section, we treat inversion of adjunction on log canonicity 
for log canonical centers. 
In order to state the main result of this section (see Theorem \ref{x-thm5.4}), 
we prepare some definitions. 

Let $(X,\Delta)$ be a normal pair such 
that $\Delta$ is effective, and let $V$ be a log canonical center 
with the normalization $V^{\nu}$. 
For any birational morphism $W\to V^{\nu}$ 
from a normal variety $W$, we define an 
$\mathbb R$-divisor $B_{W}$ on $W$ 
as follows: 
Let $f\colon Y\to X$ be a log resolution 
of $(X,\Delta)$ such that there is an 
induced surjective morphism from a 
component $T$ of $\Delta^{=1}_{Y}$ to 
$W$, where $\Delta_{Y}$ is defined by 
$K_{Y}+\Delta_{Y}=f^{*}(K_{X}+\Delta)$. 
Note that such a log resolution always exists since $V$ is 
a log canonical center. 
We put $\Delta_{T}=(\Delta_{Y}-T)|_{T}$. 
Then we obtain a projective surjective morphism 
$f_{T}\colon T\to W$, which is induced by $f\colon Y\to X$, 
such that $K_{T}+\Delta_{T}\sim_{\mathbb{R},W}0$. 
For any prime divisor $P$ on $W$ with the generic point $P_{\eta}$, we define 
$\alpha_{P,T}$ by
\begin{equation*} 
\alpha_{P,T}
=\sup\{\lambda \in \mathbb R\,|\,(T,\Delta_T+\lambda f^*_TP) 
\; \text{is sub log canonical over $P_{\eta}$}\}. 
\end{equation*} 
We note that we may assume that $P$ is a Cartier 
divisor on $W$ by shrinking $W$ suitably in the 
above definition of $\alpha_{P.T}$. 
Then we define an $\mathbb{R}$-divisor $B_{W}$ 
on $W$ by
\begin{equation*} 
B_{W}=\sum_{P}(1-\underset{T}\inf \alpha_{P,T})P 
\end{equation*} 
where $P$ runs over prime divisors on $W$ and 
$T$ 
runs over prime divisors 
over $X$ such that $a(T,X,\Delta)=-1$ 
and the image of $T$ on $X$ is $V$. 

\begin{lem}\label{x-lem5.1}
In the above notation, $B_W$ is a well-defined $\mathbb R$-divisor 
on $W$. Moreover, 
if $W=V^\nu$, then $B_{V^\nu}$ is effective. 
\end{lem}

\begin{proof}
We take a log resolution $f\colon Y\to X$ as above. 
Let $D$ be the union of the irreducible 
components of $\Delta^{=1}_Y$ that 
are dominant onto $V$ by $f$. 
Without loss of generality, by 
replacing $Y$ with a higher model, 
we may assume that the induced dominant 
rational map $D\dashrightarrow 
W$ is a morphism. 
Moreover, we may assume that there exists 
a simple normal crossing divisor $\Sigma$ on $Y$ such that 
the support of the union of $\Sigma$ and ${\rm Supp} \Delta_Y$ is a simple 
normal crossing divisor on $Y$ and 
that $D\cap \Sigma={\rm Supp} f^*_DP$, 
where $f_D:=f|_D\colon D\to W$. Let $f'\colon Y'\to X$ 
be another log resolution such that $f'\colon Y'\to X$ and $D'$ satisfies 
the same condition. If $f'\colon Y'\to X$ factors through $f\colon Y\to X$, 
then we can directly check that 
\begin{equation*}
\min _S \alpha_{P, S}=\min _{S'} \alpha_{P, S'}
\end{equation*} 
holds, where $S$ (resp.~$S'$) runs over 
irreducible components of $D$ (resp.~$D'$). 
This implies that 
\begin{equation*}
\inf_T \alpha_{P, T}=\min _S \alpha_{P, S}\in \mathbb R
\end{equation*}
holds, where $S$ runs over irreducible components of $D$. 
Hence $B_W$ is a well-defined $\mathbb R$-divisor on $W$. 

From now on, we assume that $W=V^\nu$. 
Let $E$ be the reduced $f$-exceptional divisor on $Y$. 
By running a $(K_Y+f^{-1}_*\Delta^{<1}+
{\rm Supp} f^{-1}_*\Delta^{\geq 1}+E)$-minimal model program 
over $X$, we obtain a dlt blow-up $f'\colon Y'\to X$ 
(see Theorem \ref{x-thm2.10}). 
Note that no components of $D$ are contracted in the above minimal model 
program. 
We put $K_{Y'}+\Delta_{Y'}=f'^*(K_X+\Delta)$. 
Then $\Delta_{Y'}$ is effective by construction. 
Let $D'$ be the birational transform of $D$ on $Y'$. 
Since every irreducible component of $S'$ of $D'$ is normal, 
we obtain a projective surjective morphism 
$f'_{S'}\colon S'\to V^\nu$, which is induced by $f'\colon Y'\to X$, 
such that $K_{S'}+\Delta_{S'}=(K_{Y'}+\Delta_{Y'})|_{S'}$ and 
$K_{S'}+\Delta_{S'}\sim _{\mathbb R, V^\nu}0$. 
By adjunction, $\Delta_{S'}$ is effective. 
Hence, we can easily see that 
\begin{equation*}
\inf _T \alpha_{P, T}=\min _S \alpha _{P, S}\leq 1. 
\end{equation*}
This implies that 
$B_{V^\nu}=\sum _P (1-\inf_T \alpha_{P, T})P\geq 0$. 
\end{proof}

By the above construction of $B_W$, we obtain 
an $\mathbb R$-$b$-divisor 
$\mathbf B$ 
such that $\mathbf B_{W}=B_{W}$. 
Following \cite{hacon}, we say 
that $(V^{\nu},\mathbf B)$ is 
{\em log canonical} if $(W, \mathbf B_{W})$ 
is sub log canonical for all sufficiently higher 
model $W\to V^{\nu}$, equivalently, 
all coefficients of $\mathbf B$ are not greater than one. 

\begin{rem}\label{x-rem5.2} 
If $\dim V=\dim X-1$, then $B_{V^\nu}$ is nothing but Shokurov's 
different. Moreover, by definition, we can easily check that 
$K_W+B_W=\mu^*(K_{V^{\nu}}+B_{V^\nu})$ 
holds for every proper birational morphism $\mu\colon W\to V^\nu$ from 
a normal variety $W$. 
Hence $(V^\nu, B_{V^\nu})$ is log canonical in the usual sense 
if and only if $(V^\nu, \mathbf B)$ is log canonical. 
\end{rem}

\begin{rem}\label{x-rem5.3}
Our construction of $\mathbf B$ is slightly 
different from that of 
Hacon's b-divisor $\mathbf B(V; X, \Delta)$ in 
\cite{hacon} 
because we take 
the infimum of $\alpha_{P,T}$ among $T$. 
By definition, it is clear that $\mathbf B$ 
is greater than or equal to the $b$-divisor $\mathbf B(V; X, \Delta)$ 
defined in \cite{hacon}. 
We can prove that $\mathbf B$ coincides with 
Hacon's $\mathbf B(V; X, \Delta)$. 
For the details, see \cite{fujino-hashizume-proc}. 
\end{rem}

We are ready to state the main theorem of this section. 

\begin{thm}[Log canonical inversion of adjunction, {cf.~\cite{hacon}}]
\label{x-thm5.4}
With notation as above, $(X,\Delta)$ is log 
canonical near $V$ if and only if $(V^{\nu},\mathbf B)$ is log canonical. 
\end{thm}

\begin{proof}
Since the problem is local, by shrinking $X$,  
we may assume that $X$ is quasi-projective. 
If $(X,\Delta)$ is log canonical near $V$, 
then it is easy to see that $(V^{\nu}, \mathbf B)$ is log canonical. 
Suppose that $(V^{\nu}, \mathbf B)$ is log canonical. 
By Lemma \ref{x-lem3.5}, we get a crepant 
model $f\colon (Y,\Delta_{Y})\to (X,\Delta)$ satisfying the following properties:
\begin{itemize}
\item[(i)] $a(E, X, \Delta)\leq -1$ for every $f$-exceptional 
divisor $E$ on $Y$. 
\item[(ii)] We define
\begin{equation*}
\Delta^\dag:=\Delta^{<1}_{Y}+
{\rm Supp}\Delta^{\geq1}_{Y}\qquad {\text {and}}\qquad \Gamma_{Y}:=
\Delta^{\geq1}_{Y}-{\rm Supp}\Delta^{\geq1}_{Y}. 
\end{equation*}
Then $(Y, \Delta^\dag)$ is a 
$\mathbb Q$-factorial dlt pair, $\Gamma_Y$ is effective, and the following 
equality 
\begin{equation*}
K_Y+\Delta^\dag=f^*(K_X+\Delta)-\Gamma_{Y}
\end{equation*} 
holds. 
\item[(iii)]
The divisor $K_{Y}+\Delta^{\dag}\sim_{\mathbb{R},X}-\Gamma_{Y}$ 
is semi-ample over $X$. 
\end{itemize}
Since $V$ is an lc center of $(X,\Delta)$, we 
have $f({\rm Supp}\Delta^{>1}_{Y})\not\supset V$. 
We may further assume 
that there exists a component $S$ of $\Delta^{=1}_{Y}$ 
such that $f(S)=V$. 
We note that $(V^\nu, \mathbf B)$ is log canonical 
by assumption. 
Suppose that $S\cap {\rm Supp} \Gamma_Y\ne \emptyset$ holds. 
We put $K_S+\Delta_S=(K_Y+\Delta_Y)|_S$ by adjunction. 
Since $(Y, \Delta^\dag)$ is a $\mathbb Q$-factorial dlt pair, 
$S$ is normal and ${\rm coeff}_P (\Delta_S)>1$ holds 
for every irreducible component $P$ of $S\cap {\rm Supp} \Gamma_Y$ 
(see also Lemma \ref{x-lem4.1}). 
By taking an appropriate birational model $W\to V^\nu$ of $V^\nu$, 
we may assume that the image of an irreducible component 
of $S\cap {\rm Supp} \Gamma_Y$ by the induced rational map 
$S\dashrightarrow W$ is a codimension one point of $W$. 
In this case, we can easily check $\mathbf B^{>1}_W\ne 0$. 
This is a contradiction. 
Hence we have $S\cap {\rm Supp} \Gamma_Y=\emptyset$. 

Let $g\colon Y\to Z$ be the contraction over $X$ 
induced by $K_Y+\Delta^\dag$. 
We put $\Gamma_{Z}=g_{*}\Gamma_{Y}$, and 
we put $h\colon Z\to X$ as the 
induced birational morphism. 
By construction, the morphism $(Z,g_{*}\Delta^{\dag})\to X$ 
is an lc modification of $X$ and 
$\Delta^{<1}+{\rm Supp}\Delta^{\geq1}$.  
Because $\Gamma_{Y}=g^{*}\Gamma_{Z}$, 
the divisor $-\Gamma_{Z}$ is ample over $X$ 
and ${\rm Supp}\Gamma_{Y}=g^{-1}({\rm Supp}\Gamma_{Z})$. 
This fact and $S\cap {\rm Supp}
\Gamma_{Y}=\emptyset$ imply that 
$g(S)\cap {\rm Supp}\Gamma_{Z} =\emptyset$. 
Furthermore, the inclusion 
${\rm Exc}(h)\subset {\rm Supp}\Gamma_{Z}$ holds by 
Remark \ref{x-rem3.4}. 
Therefore, $h\colon Z\to X$ is an 
isomorphism on $Z\setminus {\rm Supp}
\Gamma_{Z}$ which contains $g(S)$. 

We have proved that the lc modification 
\begin{equation*}
h\colon (Z,g_{*}\Delta^{\dag})\to (X,\Delta^{<1}+{\rm Supp}\Delta^{\geq1})
\end{equation*} 
is an isomorphism near $g(S)$. 
Since $h(g(S))=f(S)=V$, we see that $(X,\Delta)$ is log canonical near $V$.  
\end{proof}

We will treat a more precise version of adjunction and 
inversion of adjunction for log canonical centers of arbitrary 
codimension in \cite{fujino-hashizume-inversion}. 
We strongly recommend the interested reader to 
see \cite{fujino-hashizume-inversion}. 

Kawakita's inversion of adjunction on log canonicity is a 
very special case of Theorem \ref{x-thm5.4}. 

\begin{cor}[see {\cite{kawakita}}]\label{x-cor5.5}
Let $(X, S+B)$ be a normal pair such that 
$S$ is a reduced divisor, $B$ is effective, and 
$S$ and $B$ have no common irreducible components. 
Let $\nu\colon S^\nu\to S$ be the normalization of $S$. 
We put $K_{S^\nu}+B_{S^\nu}=\nu^*(K_X+S+B)$. 
Then $(X, S+B)$ is log canonical 
near $S$ if and only if $(S^\nu, B_{S^\nu})$ is log 
canonical. 
\end{cor}

\begin{proof} 
It is a direct consequence of Theorem \ref{x-thm5.4} (see 
also Remark \ref{x-rem5.2}). 
\end{proof}

\section{Proof of Theorem \ref{x-thm1.7}}\label{x-sec6}

This section is devoted to the proof of Theorem \ref{x-thm1.7}. 
Before proving Theorem \ref{x-thm1.7}, we introduce some lemmas. 

\begin{lem}\label{x-lem6.1}
Let $X$ be a normal quasi-projective 
variety and let $\Delta$ be an effective $\mathbb{R}$-divisor 
on $X$ such that $K_{X}+\Delta$ is $\mathbb{R}$-Cartier. 
Let $\pi \colon X \to S$ be a projective morphism onto 
a scheme $S$ such that $-(K_{X}+\Delta)$ is $\pi$-ample. 
Suppose that 
\begin{equation*}
\pi \colon {\rm Nklt}(X,\Delta) \to \pi({\rm Nklt}(X,\Delta))
\end{equation*} 
is finite. 
Let $g\colon (Y,\Delta_{Y})\to (X,\Delta)$ and $\Gamma_{Y}$ 
be as in Theorem \ref{x-thm1.6}. 
We consider a sequence of finite steps of 
a $(K_{Y}+\Delta_{Y}-\Gamma_{Y})$-minimal model program over $S$ 
\begin{equation*}
(Y,\Delta_{Y}-\Gamma_{Y})\dashrightarrow 
\cdots \dashrightarrow (Y',\Delta_{Y'}-\Gamma_{Y'}). 
\end{equation*}  
Let $C'\subset Y'$ be a curve contained in a 
fiber of  $Y'\to S$ such that 
$(K_{Y'}+\Delta_{Y'}-\Gamma_{Y'})\cdot C'<0$ 
and let $U$ be a Zariski open subset of $S$ containing $\pi_{Y'}(C')$, 
where $\pi_{Y'}\colon Y'\to S$. 
Suppose that 
the birational map $Y\dashrightarrow Y'$ is 
an isomorphism on an open subset containing 
\begin{equation*}
{\rm Supp}\Gamma_{Y}\cap \pi^{-1}_Y(U)=g^{-1}({\rm Nklt}(X,\Delta))
\cap \pi^{-1}_Y(U), 
\end{equation*} 
where $\pi_Y=\pi\circ g\colon Y\to S$. 
Then, the following properties hold true: 
\begin{itemize}
\item[\em{(i)}] 
${\rm Supp}\Gamma_{Y'}\not\supset C'$, and 
\item[\em{(ii)}] 
$(K_{Y'}+\Delta_{Y'})\cdot C'<0$. 
\end{itemize}
\end{lem}

\begin{proof}
If $C'\cap {\rm Supp}\Gamma_{Y'}$ is empty, 
then it is obvious that ${\rm Supp}\Gamma_{Y'}\not\supset C'$ 
and $(K_{Y'}+\Delta_{Y'})\cdot C'=(K_{Y'}+\Delta_{Y'}-\Gamma_{Y'})\cdot C'<0$ 
holds. 
Therefore, we may assume that $C'$ intersects ${\rm Supp}\Gamma_{Y'}$. 
In the argument below, we can shrink $S$ and assume that 
$S=U$. 

We take a common resolution $\phi \colon W\to Y$ 
and $\phi' \colon W \to Y'$ of the 
birational map $Y\dashrightarrow Y'$. 
Since $Y\dashrightarrow Y'$ is 
the restriction of a  
$(K_{Y}+\Delta_{Y}-\Gamma_{Y})$-minimal model 
program and 
$Y\dashrightarrow Y'$ is an isomorphism 
on an open subset containing ${\rm Supp}\Gamma_{Y}$, 
there is an effective divisor $F$ on $W$ such that
\begin{equation}\label{x-eq6.1}
\begin{split}
\phi^{*}(K_{Y}+\Delta_{Y}-\Gamma_{Y})&
=\phi'^{*}(K_{Y'}+\Delta_{Y'}-\Gamma_{Y'})+F,\quad {\text{and}}\\
\phi^{*}(K_{Y}+\Delta_{Y})&=\phi'^{*}(K_{Y'}+\Delta_{Y'})+F.
\end{split}
\end{equation}
Since $C'$ intersects ${\rm Supp}\Gamma_{Y'}$ 
and the birational map $Y\dashrightarrow Y'$ 
is an isomorphism on an open subset containing 
${\rm Supp}\Gamma_{Y}$, we see that $C'$ intersects 
an open subset $U'\subset Y'$ on which $Y'\dashrightarrow Y$ 
is an isomorphism.   
Hence we can find a curve $C_{W}$ on $W$ and a curve 
$C$ on $Y$ such that $\phi(C_{W})=C$, $\phi'(C_{W})=C'$, 
and $(F\cdot C_{W})\geq0$.
By \eqref{x-eq6.1}, we have $(K_{Y}+\Delta_{Y}-\Gamma_{Y})\cdot C
\geq (K_{Y'}+\Delta_{Y'}-\Gamma_{Y'})\cdot C'$ and 
$(K_{Y}+\Delta_{Y})\cdot C\geq (K_{Y'}+\Delta_{Y'})\cdot C'$. 
Furthermore, since $Y\dashrightarrow Y'$ is an isomorphism 
on an open subset containing ${\rm Supp}\Gamma_{Y}$, 
the condition $C'\subset {\rm Supp}\Gamma_{Y'}$ is equivalent 
to $C\subset{\rm Supp}\Gamma_{Y}$. 
From these facts, it is sufficient to show that
\begin{itemize}
\item[(a)]
${\rm Supp}\Gamma_Y\not\supset C$, and 
\item[(b)]
$(K_{Y}+\Delta_Y)\cdot C<0$. 
\end{itemize}
We recall that $g\colon Y \to X$ is the birational 
morphism as in Theorem \ref{x-thm1.6}. 
Therefore, $-\Gamma_{Y}$ is $g$-nef 
and $K_{Y}+\Delta_{Y}=g^{*}(K_{X}+\Delta)$. 
By hypothesis, $-(K_{X}+\Delta)$ is ample over $S$. 

\setcounter{step}{0}
\begin{step}\label{x-step6.1.1}
In this step, we will 
prove that $g(C)$ cannot be a point. 

Suppose by contradiction that $g(C)$ is a point. 
Then $-\Gamma_{Y}\cdot C\geq 0$ because 
$-\Gamma_{Y}$ is $g$-nef. 
On the other hand, by recalling that $C'$ 
intersects ${\rm Supp}\Gamma_{Y'}$ and 
$Y\dashrightarrow Y'$ is an isomorphism 
on an open subset containing ${\rm Supp}\Gamma_{Y}$, 
we see that $C$ intersects ${\rm Supp}\Gamma_{Y}$. 
Since ${\rm Supp}\Gamma_{Y}=g^{-1}({\rm Nklt}(X,\Delta))$, 
we have $g(C)\in {\rm Nklt}(X,\Delta)$. 
Therefore 
\begin{equation*}
C\subset g^{-1}(g(C))\subset g^{-1}({\rm Nklt}(X,\Delta))
={\rm Supp}\Gamma_{Y}. 
\end{equation*} 
This shows that $Y\dashrightarrow Y'$ is 
an isomorphism on an open subset containing $C$. 
Thus, we have 
$(K_{Y}+\Delta_{Y}-\Gamma_{Y})\cdot C
= (K_{Y'}+\Delta_{Y'}-\Gamma_{Y'})\cdot C'$.
Then we obtain 
\begin{equation*}
\begin{split}
0=g^*(K_{X}+\Delta)\cdot C&
=(K_{Y}+\Delta_{Y})\cdot C
\\ &\leq (K_{Y}+\Delta_{Y}-\Gamma_{Y})\cdot C\\
&=(K_{Y'}+\Delta_{Y'}-\Gamma_{Y'})\cdot C'<0. 
\end{split}
\end{equation*}
This is a contradiction. 
Therefore, we see that $g(C)$ cannot be a point. 
\end{step}
\begin{step}\label{x-step6.1.2}
By Step \ref{x-step6.1.1}, we may assume that $g(C)$ is a curve. 
By construction, $\pi(g(C))$ is a point. 
Let us recall that ${\rm Supp}\Gamma_{Y}=g^{-1}({\rm Nklt}(X,\Delta))$ holds 
and that
$\pi \colon {\rm Nklt}(X,\Delta) \to \pi({\rm Nklt}(X,\Delta))$ 
is finite. Therefore, if ${\rm Supp}
\Gamma_{Y}\supset C$, then $\pi(g(C))$ is not a point, 
which is a contradiction. 
Thus, we see that ${\rm Supp} 
\Gamma_{Y}\not\supset C$, which is the first property we wanted to prove. 
Since $g(C)$ is a curve, $\pi(g(C))$ 
is a point, and $-(K_{X}+\Delta)$ is 
ample over $S$, we have $(K_{X}+\Delta)\cdot g(C)<0$. 
Hence we obtain 
\begin{equation*}
\begin{split}
(K_{Y}+\Delta_{Y})\cdot C=g^*(K_{X}+\Delta)\cdot C<0, 
\end{split}
\end{equation*}
which is the second property we wanted to prove. 
\end{step}
From the above arguments, we 
obtain that ${\rm Supp}\Gamma_{Y'}\not\supset C'$ 
and $(K_{Y'}+\Delta_{Y'})\cdot C'<0$. 
We finish the proof of Lemma \ref{x-lem6.1}. 
\end{proof}

Although the following lemma is more or less well known to 
the experts, we state it here explicitly for the benefit of the reader. 

\begin{lem}[Relative Kawamata--Viehweg vanishing theorem]
\label{x-lem6.2}
Let $V$ be a normal variety 
and let $\Delta_V$ be an effective $\mathbb R$-divisor on $V$ such 
that $K_V+\Delta_V$ is $\mathbb R$-Cartier 
and that $(V, \{\Delta_V\})$ is klt. 
Let $p\colon V\to W$ be a projective surjective 
morphism between normal varieties with connected fibers. 
Assume that $-(K_V+\Delta_V)$ is $p$-ample. 
Then $R^ip_*\mathcal O_V(-\lfloor \Delta_V\rfloor)=0$ holds 
for every $i>0$. 
This implies 
that $\lfloor \Delta_V\rfloor$ is connected in a neighborhood 
of any fiber of $p$. 
In particular, if $(V, \Delta_V)$ is klt, then 
$R^ip_*\mathcal O_V=0$ for every $i>0$.  
\end{lem}
\begin{proof}
Since 
\begin{equation*} 
-\lfloor \Delta_V\rfloor -(K_V+\{\Delta_V\})=-(K_V+\Delta_V)
\end{equation*} 
is $p$-ample, we have $R^ip_*\mathcal O_V(-\lfloor \Delta_V\rfloor)=0$ 
for every $i>0$ by the relative Kawamata--Viehweg vanishing 
theorem (see \cite[Corollary 5.7.7]{fujino-foundations}). 
We consider the following short exact sequence 
\begin{equation*} 
0\to \mathcal O_V(-\lfloor \Delta_V\rfloor)\to \mathcal O_V
\to \mathcal O_{\lfloor \Delta_V\rfloor}\to 0. 
\end{equation*}  
Since $R^1p_*\mathcal O_V(-\lfloor \Delta_V\rfloor)=0$, 
we obtain the following short exact sequence: 
\begin{equation*} 
0\to p_*\mathcal O_V(-\lfloor \Delta_V\rfloor)\to 
\mathcal O_W\to p_*\mathcal O_{\lfloor \Delta_V\rfloor}\to 0. 
\end{equation*} 
This implies that ${\rm Supp} \lfloor \Delta_V\rfloor$ is connected 
in a neighborhood of any fiber of $p$. 
If we further assume that $(V, \Delta_V)$ is klt, then 
$\lfloor \Delta_V\rfloor=0$. Therefore, $R^ip_*\mathcal O_V=0$ 
for every $i>0$ when $(V, \Delta_V)$ is klt. 
\end{proof}

We are ready to prove Theorem \ref{x-thm1.7}. 

\begin{proof}[Proof of Theorem \ref{x-thm1.7}] 
By shrinking $S$ suitably, we may assume that 
$X$ and $S$ are both quasi-projective. 
Moreover, we may further assume 
that $\pi_*\mathcal O_X\simeq \mathcal O_S$ 
by taking the Stein factorization. 
By Theorem \ref{x-thm1.6}, we can construct a projective 
birational morphism $g\colon  Y\to X$ from a normal 
$\mathbb Q$-factorial variety $Y$ and 
an effective $\mathbb R$-divisor $\Gamma _Y$ 
on $Y$ satisfying (i)--(vi) in Theorem 
\ref{x-thm1.6}. Since $K_Y+\Delta_Y=g^*(K_X+\Delta)$, 
$(K_Y+\Delta_Y)|_{{\rm Nklt}(Y, \Delta_Y)}$ is nef 
over $S$ by Theorem \ref{x-thm1.6} (iv). 
Let us consider $\pi_Y:=\pi\circ g\colon  Y\to S$. 
We run a $(K_Y+\Delta_Y-\Gamma _Y)$-minimal 
model program over $S$ with scaling of an ample divisor. 
Then 
we have a sequence of flips and divisorial contractions 
\begin{equation*} 
\xymatrix{
Y=:Y_0\ar@{-->}[r]^-{\phi_0}
& Y_1\ar@{-->}[r]^-{\phi_1} & \cdots \ar@{-->}[r]^-{\phi_{i-1}}
& Y_i \ar@{-->}[r]^-{\phi_i}&\cdots 
} 
\end{equation*} 
over $S$. 
As usual, we put 
$(Y_0, \Delta_{Y_0}-\Gamma _{Y_0}):=(Y, \Delta_Y-\Gamma _Y)$, 
$\Delta_{Y_{i+1}}={\phi_i}_*\Delta_{Y_i}$, $\Gamma _{Y_{i+1}}=
{\phi_i}_*\Gamma _{Y_i}$, and $\pi_{Y_i}\colon Y_i\to S$ for 
every $i$. 

If $\dim S<\dim X$, then $K_Y+\Delta_Y-\Gamma _Y$ is not 
pseudo-effective over $S$ since $-(K_X+\Delta)$ is $\pi$-ample. 
Hence, the above minimal model program terminates at a Mori 
fiber space $p\colon (Y_k, \Delta_{Y_k}-\Gamma_{Y_k})\to Z$ 
over $S$ (see \cite{bchm}). 
 
If $\dim S=\dim X$, then $K_Y+\Delta_Y-\Gamma _Y$ is big over 
$S$ and $(Y, \Delta_Y-\Gamma _Y)$ is 
klt by (vi) in Theorem \ref{x-thm1.6}. 
Therefore, the minimal model program terminates at a good minimal 
model $(Y_k, \Delta_{Y_k}-\Gamma_{Y_k})$ over $S$ (see \cite{bchm}). 
 
\setcounter{case}{0}

\begin{case}\label{x-case1.5.1}
In this case, we assume that $\dim S=\dim X$ and 
that there exists a Zariski open neighborhood $U$ of $P$ 
such that $Y=Y_0\dashrightarrow Y_k$ is an isomorphism 
on some open subset containing ${\rm Supp} \Gamma _Y\cap 
\pi^{-1}_Y(U)$. 

Since $\dim S=\dim X$, $(Y_k, \Delta_{Y_k}-\Gamma _{Y_k})$ 
is a good minimal model 
over $S$. In particular, $K_{Y_k}+\Delta_{Y_k}-\Gamma _{Y_k}$ is 
nef over $S$. 
We can take a curve $C_0$ on $Y_0=Y$ such that 
$g(C_0)=C^\dag$ and $C_0\cap 
{\rm Supp} \Gamma _Y=C_0\cap {\rm Nklt}(Y, \Delta_Y)\ne \emptyset$. 
Since $-(K_X+\Delta)\cdot C^\dag>0$, 
$-(K_Y+\Delta_Y)\cdot C_0>0$ holds. 
Since $g(C_0)=C^\dag$, we obtain 
$C_0\not\subset {\rm Supp} \Gamma _Y$ 
because $\pi\colon {\rm Nklt}(X, \Delta)\to \pi({\rm Nklt}(X, \Delta))$ is finite and 
$\pi(C^\dag)=P$. 
Hence we have $C_0\cdot \Gamma _Y>0$. 
Therefore, $-(K_Y+\Delta_Y-\Gamma_Y)\cdot C_0>0$ holds. 
By assumption, we can easily see 
that $Y=Y_0\dashrightarrow Y_k$ is an isomorphism 
at the generic point of $C_0$. 
Thus, by the negativity lemma, we can check that 
\begin{equation*} 
0<-(K_Y+\Delta_Y-\Gamma _Y)\cdot C_0\leq -(K_{Y_k}
+\Delta_{Y_k}-\Gamma_{Y_k})\cdot C_k
\end{equation*} 
holds, where $C_k$ is the strict transform of $C_0$ on $Y_k$. 
This is a contradiction because $K_{Y_k}+\Delta_{Y_k}-\Gamma_{Y_k}$ 
is nef over $S$. Hence this case never happens. 
\end{case}

\begin{case}\label{x-case1.5.2}
In this case, we assume that $\dim S<\dim X$ and 
that there exists a Zariski open neighborhood $U$ of $P$ 
such that $Y=Y_0\dashrightarrow Y_k$ is an isomorphism 
on some open subset containing ${\rm Supp} \Gamma _Y\cap 
\pi^{-1}_Y(U)$. 

Since $\dim S<\dim X$, the $(K_Y+\Delta_Y-\Gamma_Y)$-minimal 
model program 
terminates at a Mori fiber space $p\colon 
(Y_k, \Delta_{Y_k}-\Gamma_{Y_k}) \to Z$ over $S$. 
\begin{equation*} 
\xymatrix{
Y=Y_0\ar@{-->}[r]^-{\phi_0}\ar[ddr]_-{\pi_Y=\pi\circ g}
& Y_1\ar@{-->}[r]^-{\phi_1} & \cdots \ar@{-->}[r]^-{\phi_{k-1}}
& Y_k\ar[d]^-p\\
& & &  Z\ar[lld]^-{\pi_Z} \\ 
& S &&
}
\end{equation*}  
We note that $P\in \pi_{Y_k}({\rm Supp} \Gamma_{Y_k})$ 
since $P\in \pi({\rm Nklt}(X, \Delta)) =\pi_Y({\rm Supp} \Gamma_Y)$. 
Hence we can take a curve $C_k$ on $Y_k$ such that 
$p(C_k)$ is a point, $\pi_{Y_k}(C_k)=P$, 
and $C_k\cap {\rm Supp} \Gamma _{Y_k}\ne \emptyset$. 
Then, by Lemma \ref{x-lem6.1}, 
$-(K_{Y_k}+\Delta_{Y_k})\cdot C_k>0$ 
and $C_k \not \subset {\rm Supp} \Gamma_{Y_k}$. 
In particular, $\Gamma _{Y_k}$ is $p$-ample. 
Since $(Y_k, \Delta_{Y_k}-\Gamma_{Y_k})$ is klt and 
$-(K_{Y_k}+\Delta_{Y_k}-\Gamma_{Y_k})$ is $p$-ample, 
we have $R^ip_*\mathcal O_{Y_k}=0$ for every $i>0$ 
by Lemma \ref{x-lem6.2}. We put $\pi_Z\colon Z\to S$. 
Since 
\begin{equation*} 
-\lfloor 
\Delta_{Y_k}\rfloor -(K_{Y_k}+\{\Delta_{Y_k}\})=-(K_{Y_k}+\Delta_{Y_k})
\end{equation*}  
is $p$-ample and $(Y_k, \{\Delta_{Y_k}\})|_{\pi^{-1}_{Y_k}(U)}$ is klt, we 
obtain that 
$R^ip_*\mathcal O_{Y_k}(-\lfloor \Delta_{Y_k}\rfloor)=0$ holds 
on $\pi^{-1}_Z(U)$ 
for 
every $i>0$ and that 
${\rm Supp} \lfloor \Delta_{Y_k}\rfloor ={\rm Supp} \Gamma_{Y_k}$ is connected 
in a neighborhood of any fiber of $p$ on $\pi^{-1}_Z(U)$ 
by Lemma \ref{x-lem6.2}. 
By Lemma \ref{x-lem6.1}, we see that 
${\rm Supp} \Gamma_{Y_k} \cap \pi^{-1}_{Y_k}(U)$ is finite 
over $\pi^{-1}_Z(U)$. 
Hence, as in Case 1 in the proof of \cite[Proposition 9.1]{fujino}, 
$\dim p^{-1}(z)=1$ for every closed point $z\in \pi^{-1}_Z(U)$. 
Then, by \cite[Lemma 8.2]{fujino}, $C_k\simeq \mathbb P^1$, 
$C_k\cap {\rm Supp} \Gamma_{Y_k}$ is a point, and 
$0 < -(K_{Y_k}+\Delta_{Y_k})\cdot C_k\leq 1$ holds. 
By using the negativity lemma, 
we can check that 
\begin{equation*} 
-(K_{Y_0}+\Delta_{Y_0})\cdot C_0\leq 
-(K_{Y_k}+\Delta_{Y_k})\cdot C_k\leq 1
\end{equation*} 
holds, where $C_0$ is the strict transform of $C_k$ on $Y_0=Y$. 
Note that $C_0\cap {\rm Nklt}(Y_0, \Delta_{Y_0})=C_0\cap 
{\rm Supp} \Gamma_Y$ 
is a point since $Y=Y_0\dashrightarrow Y_k$ is 
an isomorphism in a neighborhood of 
${\rm Supp} \Gamma _Y\cap \pi^{-1}_Y(U)$. 
Therefore, $C=g(C_0)$ is a curve on $X$ such that 
$C\cap {\rm Nklt}(X, \Delta)$ is a point by Theorem 
\ref{x-thm1.6} (iv) with 
$0<-(K_X+\Delta)\cdot C\leq 1$. 
Hence we can construct a morphism 
\begin{equation*} 
f\colon \mathbb A^1\longrightarrow \left(X\setminus {\rm Nklt}(X, \Delta)\right)
\cap \pi^{-1}(P)
\end{equation*} 
such that $f(\mathbb A^1)= C\cap (X\setminus {\rm Nklt}(X, \Delta))$. 
This is a desired morphism. 
\end{case}
\begin{case}\label{x-case1.5.3} 
By Cases \ref{x-case1.5.1} and \ref{x-case1.5.2}, 
it is sufficient to treat the following situation. 
There exist a Zariski open neighborhood $U$ of $P$ and 
$m\geq 0$ such that 
\begin{itemize}
\item[(i)] for any $i\leq m$, the map 
$Y\dashrightarrow Y_i$ is an isomorphism on some 
open subset containing ${\rm Supp} \Gamma_Y\cap \pi^{-1}_Y(U)$, and 
\item[(ii)] there is a curve $C'\subset Y_m$ contracted by the extremal 
birational contraction of the $(K_Y+\Delta_Y-\Gamma_Y)$-minimal model 
program over $S$ such that 
$C'\cap {\rm Supp} \Gamma_{Y_m}\ne \emptyset$ and 
$\pi_{Y_m}(C')=P$. 
\end{itemize}
Essentially the same argument 
as in Case \ref{x-case1.5.2} 
above works with some minor modifications. 
Let us see it more precisely. 
Let $\varphi\colon Y_m\to Z$ be the extremal birational 
contraction in (ii). Let $\pi_Z\colon Z\to S$ be the structure 
morphism. Then, by Lemma \ref{x-lem6.1}, $\Gamma_{Y_m}$ is ample 
over $\pi^{-1}_Z(U)$ and ${\rm Supp} \Gamma _{Y_m}\cap 
\pi^{-1}_{Y_m}(U)$ is finite over $\pi^{-1}_Z(U)$. 
By Lemmas \ref{x-lem6.1} and \ref{x-lem6.2}, 
we see that ${\rm Supp} \Gamma _{Y_m}$ 
is connected in a neighborhood of any fiber of $\varphi$ 
on $\pi^{-1}_Z(U)$. 
Therefore, $C'\cap {\rm Supp} \Gamma _{Y_m}$ is a point. 
By Lemma \ref{x-lem6.1} again, $\dim \varphi^{-1}(z)\leq 1$ holds 
for every closed point $z\in \pi^{-1}_Z(U)$. 
By Lemma \ref{x-lem6.2} , 
$R^i\varphi_*\mathcal O_{Y_m}=0$ 
holds on $\pi^{-1}_Z(U)$ for every $i>0$. 
Thus, by \cite[Lemma 8.2]{fujino}, 
$C'\simeq \mathbb P^1$ with $-(K_{Y_m}+\Delta_{Y_m})\cdot C'\leq 1$. 
By the negativity lemma, 
we can check that 
\begin{equation*}
-(K_{Y_0}+\Delta_{Y_0})\cdot C_0 
\leq -(K_{Y_m}+\Delta_{Y_m})\cdot C'\leq 1
\end{equation*}
holds, where $C_0$ is the strict transform of $C'$ on $Y_0=Y$. 
We note that 
$C_0\cap {\rm Nklt}(Y_0, \Delta_{Y_0})=C_0\cap 
{\rm Supp}\Gamma _Y$ is a point since $Y=Y_0\dashrightarrow 
Y_m$ is an isomorphism in a neighborhood of ${\rm Supp}\Gamma _Y 
\cap \pi^{-1}_Y(U)$.  
Hence, by the same argument as in Case \ref{x-case1.5.2} above, we get 
a desired morphism 
\begin{equation*} 
f\colon  \mathbb A^1\longrightarrow 
\left(X\setminus {\rm Nklt}(X, \Delta)\right)\cap 
\pi^{-1}(P). 
\end{equation*} 
\end{case}
We finish the proof of Theorem \ref{x-thm1.7}. 
\end{proof}

We close this section with the following generalization of 
\cite[Theorem 9.2]{fujino}. We will use it in the proof of 
Theorem \ref{x-thm1.8}.  

\begin{thm}\label{x-thm6.3}
Let $\pi\colon X\to S$ be a proper surjective morphism 
from a normal quasi-projective variety $X$ onto a scheme 
$S$. 
Let $\mathcal P$ be an $\mathbb R$-Cartier 
divisor on $X$ and let $H$ be an ample Cartier divisor 
on $X$. 
Let $\Sigma$ be a closed subset of $X$ and let 
$P$ be a closed point of $S$ such that 
there exists a curve $C^\dag\subset \pi^{-1}(P)$ with $\Sigma 
\cap C^\dag\ne \emptyset$. 
Assume that $-\mathcal P$ is $\pi$-ample and 
that $\pi\colon  \Sigma\to \pi(\Sigma)$ is finite. 
We further assume 
\begin{itemize}
\item[\em{(i)}] $\{\varepsilon _i\}_{i=1}^\infty$ is 
a set of positive real numbers with $\varepsilon _i\searrow 
0$ for $i\nearrow \infty$, and 
\item[\em{(ii)}] for every $i$, there exists an effective $\mathbb R$-divisor 
$\Delta_i$ on $X$ such that 
\begin{equation*} 
\mathcal P+\varepsilon _i H\sim _{\mathbb R} K_X+\Delta_i
\end{equation*} 
and that 
\begin{equation*} 
\Sigma ={\rm Nklt}(X, \Delta_i)
\end{equation*} 
holds set theoretically. 
\end{itemize} 
Then there exists a non-constant morphism 
\begin{equation*} 
f\colon \mathbb A^1\longrightarrow (X\setminus \Sigma)\cap 
\pi^{-1}(P) 
\end{equation*} 
such that 
the curve $C$, the closure of $f(\mathbb A^1)$ in $X$, 
is a rational curve with 
\begin{equation*} 
0< -\mathcal P\cdot C\leq 1
\end{equation*} 
\end{thm}

\begin{proof}
The proof of \cite[Theorem 9.2]{fujino} works as well in this case 
by replacing \cite[Theorem 1.8]{fujino} in the 
proof of \cite[Theorem 9.2]{fujino} with Theorem \ref{x-thm1.7}. 
\end{proof}

\section{Quick review of quasi-log schemes}\label{x-sec7}

In this section, we collect some basic definitions 
of the theory of quasi-log schemes. 
For the details, see \cite[Chapter 6]{fujino-foundations} and 
\cite{fujino}. 
Let us start with the definition of {\em{globally embedded simple normal 
crossing pairs}}.  

\begin{defn}[{Globally embedded simple normal crossing 
pairs, see \cite[Definition 6.2.1]{fujino-foundations}}]\label{x-def7.1} 
Let $Y$ be a simple normal crossing divisor 
on a smooth 
variety $M$ and let $B$ be an $\mathbb R$-divisor 
on $M$ such that ${\rm Supp} (B+Y)$ is a simple 
normal crossing divisor on $M$ and that 
$B$ and $Y$ have no common irreducible components. 
We put $B_Y=B|_Y$ and consider the pair $(Y, B_Y)$. 
We call $(Y, B_Y)$ a {\em{globally embedded simple normal 
crossing pair}} and $M$ the {\em{ambient space}} 
of $(Y, B_Y)$. A {\em{stratum}} of $(Y, B_Y)$ is a log canonical  
center of $(M, Y+B)$ that is contained in $Y$. 
\end{defn}

Let us recall the definition of {\em{quasi-log schemes}}. 

\begin{defn}[{Quasi-log 
schemes, see \cite[Definition 6.2.2]{fujino-foundations}}]\label{x-def7.2}
A {\em{quasi-log scheme}} is a scheme $X$ endowed with an 
$\mathbb R$-Cartier divisor 
(or $\mathbb R$-line bundle) 
$\omega$ on $X$, a closed subscheme 
$X_{-\infty}\subsetneq X$, and a finite collection $\{C\}$ of reduced 
and irreducible subschemes of $X$ such that there is a 
proper morphism $f\colon (Y, B_Y)\to X$ from a globally 
embedded simple 
normal crossing pair satisfying the following properties: 
\begin{itemize}
\item[(1)] $f^*\omega\sim_{\mathbb R}K_Y+B_Y$. 
\item[(2)] The natural map 
$\mathcal O_X
\to f_*\mathcal O_Y(\lceil -(B_Y^{<1})\rceil)$ 
induces an isomorphism 
\begin{equation*} 
\mathcal I_{X_{-\infty}}\overset{\simeq}{\longrightarrow} 
f_*\mathcal O_Y(\lceil 
-(B_Y^{<1})\rceil-\lfloor B_Y^{>1}\rfloor),  
\end{equation*} 
where $\mathcal I_{X_{-\infty}}$ is the defining ideal sheaf of 
$X_{-\infty}$. 
\item[(3)] The collection of reduced and irreducible subschemes 
$\{C\}$ coincides with the images 
of the strata of $(Y, B_Y)$ that are not included in $X_{-\infty}$. 
\end{itemize}
We simply write $[X, \omega]$ to denote 
the above data 
\begin{equation*} 
\left(X, \omega, f\colon (Y, B_Y)\to X\right)
\end{equation*} 
if there is no risk of confusion. 
Note that a quasi-log scheme $[X, \omega]$ is 
the union of $\{C\}$ and $X_{-\infty}$. 
The reduced and irreducible subschemes $C$ 
are called the {\em{qlc strata}} of $[X, \omega]$, 
$X_{-\infty}$ is called the {\em{non-qlc locus}} 
of $[X, \omega]$, and $f\colon (Y, B_Y)\to X$ is 
called a {\em{quasi-log resolution}} 
of $[X, \omega]$. 
We sometimes use ${\rm Nqlc}(X, 
\omega)$ or 
\begin{equation*} 
{\rm Nqlc}(X, \omega, f\colon (Y, B_Y)\to X)
\end{equation*} 
to denote 
$X_{-\infty}$. 
If a qlc stratum $C$ of $[X, \omega]$ is not an 
irreducible component of $X$, then 
it is called a {\em{qlc center}} of $[X, \omega]$. 
\end{defn}

\begin{defn}[Open qlc strata]\label{x-def7.3}
Let $W$ be a qlc stratum of a quasi-log scheme $[X, \omega]$. 
We put 
\begin{equation*} 
U:=W\setminus \left\{ \left(W\cap {\rm Nqlc}(X, \omega)\right)\cup 
\bigcup_{W'}W'\right\}, 
\end{equation*} 
where $W'$ runs over qlc centers of $[X, \omega]$ strictly contained 
in $W$, and call it the {\em{open qlc stratum}} of 
$[X, \omega]$ associated to $W$. 
\end{defn}

\begin{defn}[${\rm Nqklt}(X, \omega)$]\label{x-def7.4}
Let $[X, \omega]$ be a quasi-log scheme. 
The union of ${\rm Nqlc}(X, \omega)$ and all qlc centers of $[X, \omega]$ 
is denoted by ${\rm Nqklt}(X, \omega)$. 
Note that if ${\rm Nqklt}(X, \omega)\ne {\rm Nqlc}(X, \omega)$ then 
$[{\rm Nqklt}(X, \omega), \omega|_{{\rm Nqklt}(X, \omega)}]$ 
naturally becomes a quasi-log scheme by adjunction 
(see \cite[Theorem 6.3.5 (i)]{fujino-foundations} 
and \cite[Theorem 4.6 (i)]{fujino}). 
\end{defn}

Although we do not treat applications 
of the theory of quasi-log schemes to normal pairs here, 
the following remark is very important. 

\begin{rem}\label{x-rem7.5}
Let $(X, \Delta)$ be a normal pair such that 
$\Delta$ is effective. 
Then $[X, K_X+\Delta]$ naturally becomes a quasi-log scheme such 
that ${\rm Nqlc}(X, K_X+\Delta)$ coincides with 
${\rm Nlc}(X, \Delta)$ and that 
$C$ is a qlc center of $[X, K_X+\Delta]$ 
if and only if $C$ is a log canonical center of $(X, \Delta)$. 
Hence ${\rm Nqklt}(X, K_X+\Delta)$ corresponds to 
${\rm Nklt}(X, \Delta)$. 
For the details, see \cite[6.4.1]{fujino-foundations} and 
\cite[Example 4.10]{fujino}. 
\end{rem}

\section{Proof of Theorems \ref{x-thm1.8} and 
\ref{x-thm1.9}}\label{x-sec8} 

In this section, we prove Theorems \ref{x-thm1.8} and 
\ref{x-thm1.9}. Let us start with the proof of Theorem \ref{x-thm1.8}. 

\begin{proof}[Proof of Theorem \ref{x-thm1.8}]
By Steps 1, 2, 3, and 4 in 
the proof of \cite[Theorem 1.6]{fujino}, 
we can reduce the problem to the case where 
$X$ is a normal variety such that 
$-\omega$ is $\pi$-ample and that 
$\pi\colon {\rm Nqklt}(X, \omega)\to \pi({\rm Nqklt}(X, \omega))$ is finite. 
By taking the Stein factorization, we may further assume that 
$\pi_*\mathcal O_X\simeq \mathcal O_S$. 
We put $\Sigma={\rm Nqklt}(X, \omega)$. 
It is sufficient to find a non-constant morphism 
\begin{equation*} 
f\colon \mathbb A^1\longrightarrow (X\setminus \Sigma)\cap 
\pi^{-1}(P)
\end{equation*} 
such that the curve $C$, the closure of $f(\mathbb A^1)$ in $X$, 
is a (possibly singular) rational curve satisfying 
$C\cap \Sigma\ne \emptyset$ with 
\begin{equation*} 
0<-\omega\cdot C\leq 1. 
\end{equation*}  
Without loss of generality, we may assume that 
$X$ and $S$ are quasi-projective by shrinking $S$ suitably. 
Hence we have the following properties: 
\begin{itemize}
\item[(a)] $\pi\colon X\to S$ is a projective morphism 
from a normal quasi-projective variety $X$ to a scheme $S$, 
\item[(b)] $-\omega$ is $\pi$-ample, and 
\item[(c)] $\pi\colon \Sigma\to \pi(\Sigma)$ is finite, 
where $\Sigma: ={\rm Nqklt}(X, \omega)$. 
\end{itemize}
Let $H$ be an ample Cartier divisor on $X$ 
and let $\{\varepsilon _i \}_{i=1}^\infty$ be 
a set of positive real numbers such 
that $\varepsilon _i\searrow 0$ for $i\nearrow \infty$. 
Then, by \cite[Theorem 1.10]{fujino}, 
we have: 
\begin{itemize}
\item[(d)] there exists an effective $\mathbb R$-divisor 
$\Delta_i$ on $X$ such that 
\begin{equation*} 
K_X+\Delta_i\sim _{\mathbb R} \omega+\varepsilon _i H
\end{equation*} 
with 
\begin{equation*} 
{\rm Nklt}(X, \Delta_i)=\Sigma
\end{equation*} 
for every $i$. 
\end{itemize}
Thus, by Theorem \ref{x-thm6.3}, 
we have a desired non-constant morphism 
\begin{equation*} 
f\colon  \mathbb A^1\longrightarrow 
\left(X\setminus {\rm Nqklt}(X, \omega)\right)
\cap \pi^{-1}(P). 
\end{equation*} 
We complete the proof. 
\end{proof}

Finally, we prove Theorem \ref{x-thm1.9}. 

\begin{proof}[Proof of Theorem \ref{x-thm1.9}]
We put $X'=\overline {U_j}\cup {\rm Nqlc}(X, \omega)$. 
Then $[X', \omega']$ naturally becomes a quasi-log 
scheme by adjunction, 
where $\omega'=\omega|_{X'}$ (see 
\cite[Theorem 6.3.5 (i)]{fujino-foundations} and 
\cite[Theorem 4.6 (i)]{fujino}). 
The induced morphism 
$\varphi_{R_j}\colon X'\to \varphi_{R_j}(X')$ 
is denoted by $\pi'\colon X'\to S'$. 
Then, $-\omega'$ is $\pi'$-ample, 
\begin{equation*} 
\pi'\colon {\rm Nqklt}(X', \omega')\to \pi'({\rm Nqklt}(X', \omega')) 
\end{equation*} 
is finite, and there is a curve $C^\dag\subset (\pi')^{-1}(P)$ with 
${\rm Nqklt}(X', \omega')\cap C^\dag\ne \emptyset$. 
Hence, by Theorem \ref{x-thm1.8}, 
there exists a non-constant morphism 
\begin{equation*} 
f_j\colon \mathbb A^1\longrightarrow U_j \cap \varphi^{-1}_{R_j}(P)
\end{equation*} 
with the desired properties. 
\end{proof}

\begin{rem}\label{x-rem8.1}
We use the same notation as in the proof of 
Theorem \ref{x-thm1.9}. 
Since $\varphi:=\varphi_{R_j}\colon X\to V:=\varphi_{R_j}(X)$ 
is a contraction morphism associated to $R_j$, 
the natural isomorphism $\varphi_*\mathcal O_X
\simeq \mathcal O_V$ holds (see 
\cite[Theorem 6.7.3 (ii)]{fujino-foundations} and 
\cite[Theorem 4.17 (ii)]{fujino}). 
Let $\mathcal I_{X'}$ be the defining ideal sheaf of $X'$ on $X$. 
Then, by the vanishing theorem (see 
\cite[Theorem 6.3.5 (ii)]{fujino-foundations} and 
\cite[Theorem 4.6 (ii)]{fujino}), 
we have $R^i\varphi_*\mathcal I_{X'}=0$ for 
every $i>0$ since $-\omega$ is $\varphi$-ample. 
Thus we obtain the following short exact sequence 
\begin{equation*} 
0\to \varphi_*\mathcal I_{X'}\to \varphi_*\mathcal O_X
\simeq \mathcal O_V\to \varphi_*\mathcal O_{X'}\to 0. 
\end{equation*} 
This means that $\varphi_{R_j}\colon X'\to \varphi_{R_j}(X')$ 
has connected fibers. 
Therefore, if $Q$ is a close point of $\pi'\left({\rm Nqklt}(X', 
\omega')\right)$ with $\dim \pi'^{-1}(Q)\geq 1$, then 
we can always find a curve $\widetilde C$ such that 
$\varphi_{R_j}(\widetilde C)=Q$, $\widetilde C\not\subset 
U_j$, and $\widetilde C\subset \overline {U_j}$. 
\end{rem}

%%%%%%%%%%%%%%%

\end{document}